\documentclass[11pt]{amsart}

\usepackage{graphics}
\usepackage[dvips]{epsfig}
\usepackage{psfrag}

\usepackage{amsmath}
\usepackage{amsfonts}
\usepackage{latexsym}
\usepackage{amssymb}
\usepackage[usenames]{color}
\usepackage{amsthm}
\usepackage{pinlabel}
\usepackage{textcomp}

\usepackage[all]{xy}

\theoremstyle{plain}
\newtheorem{defn}{Definition}[section]
\newtheorem{Thm}[defn]{Theorem}
\newtheorem{Prop}[defn]{Proposition}
\newtheorem{Lem}[defn]{Lemma}
\newtheorem{Cor}[defn]{Corollary}
\theoremstyle{remark}
\newtheorem{Rem}[defn]{Remark}

\newtheorem{Ex}[defn]{Example}

\numberwithin{equation}{section}

\newcommand{\R}{\ensuremath{\mathbb{R}}}
\newcommand{\Z}{\ensuremath{\mathbb{Z}}}

\begin{document}
\title[Geography of bilinearized LCH]{Geography of bilinearized Legendrian contact homology}
\author[F. Bourgeois]{Fr\'ed\'eric Bourgeois} \address{Universit\'e Paris-Saclay, CNRS, Laboratoire de Math\'ematiques d'Orsay, 
91405 Orsay, France} \email{frederic.bourgeois@universite-paris-saclay.fr} 
\urladdr{https://www.imo.universite-paris-saclay.fr/{\raisebox{0.5ex}{\texttildelow}}bourgeois/}
\author[D. Galant]{Damien Galant} \address{D\'epartement de Math\'ematiques, 
University of MONS (UMONS), Place du Parc 20, 7000 Mons, Belgium}
\address{
Universit\'e Polytechnique Hauts-de-France, INSA Hauts-de-France, 
CERAMATHS - Laboratoire de Mat\'eriaux C\'eramiques et de Math\'ematiques, 
F-59313 Valenciennes, France}
\email{damien.galant@umons.ac.be}
\urladdr{http://math.umons.ac.be/staff/Galant.Damien/}

\begin{abstract}
We study the geography of bilinearized Legendrian contact homology for closed, connected Legendrian
submanifolds with vanishing Maslov class in $1$-jet spaces. We show that this invariant detects whether
the two augmentations used to define it are DGA homotopic or not. We describe a collection of graded
vector spaces containing all possible values for bilinearized Legendrian contact homology and then show
that all these vector spaces can be realized.
\end{abstract} 

\maketitle

\section{Introduction}
\label{sec:introduction}

Let $\Lambda$ be a closed Legendrian submanifold of the $1$-jet space $J^1(M)$ of a manifold $M$.
Given a generic complex structure for the canonical contact structure on $J^1(M)$, one can associate to $\Lambda$
its Chekanov-Eliashberg differential graded algebra $(\mathcal{A}(\Lambda), \partial)$, see~\cite{C, EES1, EES3}.
The homology of $(\mathcal{A}(\Lambda), \partial)$, called Legendrian contact homology, is an invariant of the Legendrian 
isotopy class of $\Lambda$,
but it is often hard to compute. It is therefore useful to consider augmentations of $(\mathcal{A}(\Lambda), \partial)$,
because such an augmentation $\varepsilon$ can be used to define a linearized complex $(C(\Lambda), \partial^\varepsilon)$.
The homology is denoted by $LCH^\varepsilon(\Lambda)$ and called linearized Legendrian contact homology of $\Lambda$
with respect to $\varepsilon$. The collection of these homologies for all augmentations of $(\mathcal{A}(\Lambda), \partial)$
is also an invariant of the Legendrian isotopy class of $\Lambda$. The geography (i.e. the determination of all possible values) 
of a similar homological invariant defined using generating families was described by the first author with Sabloff and 
Traynor~\cite{BST}. Using the work of Dimitroglou Rizell~\cite{DR} on the effect of embedded surgeries on Legendrian 
contact homology, this geography can be shown to hold for linearized Legendrian contact homology as well. 
On the other hand, the first author and Chantraine showed~\cite{BC}  that it is possible
to use two augmentations $\varepsilon_1, \varepsilon_2$ of the Chekanov-Eliashberg DGA in order to define a bilinearized
differential $\partial^{\varepsilon_1, \varepsilon_2}$ on $C(\Lambda)$. The corresponding homology is called bilinearized 
Legendrian contact homology and is denoted by $LCH^{\varepsilon_1, \varepsilon_2}(\Lambda)$. The object of this article
is to describe the geography of bilinearized Legendrian contact homology. In other words, our goal is to describe a collection
of Legendrian submanifolds equipped with two augmentations such that their bilinearized Legendrian contact homologies
realize all possible values for this invariant.

When $\varepsilon_1 = \varepsilon_2$, bilinearized Legendrian contact homology coincides with linearized Legendrian contact 
homology. More generally, if the two augmentations are DGA homotopic, $LCH^{\varepsilon_1, \varepsilon_2}(\Lambda)$ is
isomorphic to $LCH^{\varepsilon_1}(\Lambda)$. Our first result describes a crucial difference in the behavior of bilinearized 
Legendrian contact homology depending whether the two augmentations are DGA homotopic or not. More precisely, this
different behavior is detected by a map $\tau_n : LCH^{\varepsilon_1, \varepsilon_2}_n(\Lambda) \to H_n(\Lambda)$
appearing in the duality exact sequence for Legendrian contact homology~\cite{EESdual} and described in Sections~\ref{sec:bLCH} and~\ref{sec:fundclass}.

\begin{Thm} \label{thm:A}
Let $\Lambda$ be a closed, connected Legendrian submanifold of $J^1(M)$ with $\dim M = n$. 
Let $\varepsilon_1, \varepsilon_2$ be two augmentations of the Chekanov-Eliashberg DGA of $\Lambda$
with coefficients in $\Z_2$. 
Then $\varepsilon_1$ and $\varepsilon_2$ are DGA homotopic if and only if the map 
$\tau_n : LCH^{\varepsilon_1, \varepsilon_2}_n(\Lambda) \to H_n(\Lambda)$ is surjective.
\end{Thm}

In other words, this means that the fundamental class of $\Lambda$ induces a class in linearized 
Legendrian contact homology, while the class of the point in $\Lambda$ induces a class in
bilinearized Legendrian contact homology with respect to non DGA homotopic augmentations.

\begin{Cor} \label{cor:A}
Bilinearized Legendrian contact homology is a complete invariant for DGA homotopy classes of
augmentations of the Chekanov-Eliashberg DGA.
\end{Cor}

The strength of this result will be illustrated in Section~\ref{sec:fundclass} by revisiting an important example 
of Legendrian knot featuring only a partial study of its augmentations~\cite{MS}. In this paper we complete the study 
of this Legendrian knot with a full description of its DGA homotopy classes of augmentations.

Our second result describes the geography of the Laurent polynomials that can be obtained
as a Poincar\'e polynomial for bilinearized Legendrian contact homology. 
We will introduce in Definition~\ref{def:bLCHadm} the explicit notion of bLCH-admissible Laurent polynomial,
and prove that only these polynomials can be obtained as the Poincar\'e polynomial of bilinearized 
Legendrian contact homology.

\begin{Thm}  \label{thm:B}
For any bLCH-admissible Laurent polynomial $P$, there exists a closed, connected Legendrian submanifold $\Lambda$ 
of $J^1(M)$ and there exist two non DGA homotopic augmentations $\varepsilon_1, \varepsilon_2$ of the 
Chekanov-Eliashberg DGA of $\Lambda$, with the property that the Poincar\'e polynomial of 
$LCH^{\varepsilon_1, \varepsilon_2}(\Lambda)$ with coefficients in $\mathbb{Z}_2$ is equal to $P$.
\end{Thm}

We also will establish a similar result, namely Theorem~\ref{thm:geogsphere}, in the specific case of Legendrian spheres.

The collection of Poincar\'e polynomials that is realized by bilinearized Legendrian contact homology is considerably
wider than the corresponding collection for the geography of linearized Legendrian contact homology~\cite[Theorem~1.1]{BST}.
For this reason, the examples of Legendrian submanifolds that are constructed in this paper in order to realize the geography
of bilinearized Legendrian contact homology differ substantially from those considered in~\cite{BST} and exhibit
new interesting phenomena. In particular, while connected sums of Legendrian submanifolds played an important role in~\cite{BST},
such constructions cannot be used in this paper because these tend to produce pairs of unwanted generators in bilinearized 
Legendrian contact homology. Moreover, we introduce a completely new construction in order to create pairs of generators in 
arbitrary degrees, instead of degrees summing to $n-1$ as in linearized Legendrian contact homology. We also introduce 
another completely new construction in order to obtain bilinearized Legendrian contact homologies of different ranks, 
depending on the ordering of the two non DGA homotopic augmentations. Note that the examples 
constructed in this paper are convenient to work with, as they only have cusp singularities.

This paper is organized as follows. In Section~\ref{sec:bLCH} we review the definition of bilinearized Legendrian 
contact homology and state its main properties. In Section~\ref{sec:fundclass} we study fundamental classes
in bilinearized Legendrian contact homology, prove Theorem~\ref{thm:A} and Corollary~ \ref{cor:A} and study the
effect of connected sums on bilinearized Legendrian contact homology.
In Section~\ref{sec:geog}, we study the geography of bilinearized Legendrian contact homology and
prove Theorems~\ref{thm:B} and its counterpart Theorem~\ref{thm:geogsphere} for Legendrian spheres.

{\bf Acknowledgements. } 
We are indebted to Josh Sabloff for providing us a computer code that computes linearized Legendrian contact homology of 
Legendrian knots in $\R^3$, using techniques presented in~\cite{HR1} and \cite{HR2}.  Although the exposition of this paper
is independent from these sources, the generalization of this computer code by DG to the calculation of bilinearized Legendrian contact homology played an essential role at the beginning of this work, before its generalization to higher dimensions. 
We thank Georgios Dimitroglou Rizell for a productive discussion of DGA homotopies
of augmentations.  
An important refinement in the constructions from Section~\ref{sec:geog} emerged after an 
interesting conversation with Sylvain Courte.
Special thanks go to Filip Strako\v{s} for spotting a mistake in the proof of Proposition~\ref{prop:bLCHadm} impacting
other parts of an earlier version of the paper.
We also thank Cyril Falcon for his remarks on the original manuscript.
FB was partially supported by the Institut Universitaire de France and by the ANR projects 
Quantact (16-CE40-0017) and Microlocal (15-CE40-0007).

\section{Bilinearized Legendrian contact homology}
\label{sec:bLCH}

The $1$-jet space $J^1(M) = T^*M \times \mathbb{R}$ of a smooth, $n$-dimensional manifold $M$ is equipped
with a canonical contact structure $\xi = \ker(dz - \lambda)$, where $\lambda$ is the Liouville $1$-form on $T^*M$
and $z$ is the coordinate along $\mathbb{R}$. Let $\Lambda$ be a closed Legendrian submanifold of this contact
manifold, i.e. a closed, embedded submanifold of dimension $n$ such that $T_p\Lambda \subset \xi_p$ for any
$p \in \Lambda$. 

We first describe the definition of a differential graded algebra associated to $\Lambda$, following its construction
by Ekholm, Etnyre and Sullivan~\cite{EES1}. The Reeb vector field associated to the contact form 
$\alpha = dz -\lambda$ for $\xi$ is simply
$R_\alpha = \frac{\partial}{\partial z}$. A Reeb chord of $\Lambda$ is a finite, nontrivial piece of integral curve 
for $R_\alpha$ with endpoints on $\Lambda$. After performing a Legendrian isotopy, we can assume that all Reeb
chords of $\Lambda$ are nondegenerate, i.e. the canonical projections to the tangent space of $T^*M$ of the 
tangent spaces to $\Lambda$ at the endpoints of each chord are intersecting transversally. Let us assume that
the Maslov class $\mu(\Lambda)$ of $\Lambda$ vanishes, see~\cite[section 2.2]{EES1}.

We denote by $\mathcal{A}(\Lambda)$ the unital, noncommutative algebra freely generated over $\mathbb{Z}_2$ 
by the Reeb chords of $\Lambda$. Each Reeb chord $c$ is graded by its Conley-Zehnder $\nu(c) \in \mathbb{Z}$; 
when $\Lambda$ is connected, this does not depend on any additional choice since $\mu(\Lambda)=0$. The grading of 
$c$ is defined as $|c| = \nu(c) - 1$. Hence, in this case the algebra $\mathcal{A}(\Lambda)$ is naturally graded.

Let $J$ be a complex structure on $\xi$, which is compatible with its conformal symplectic structure. This complex structure
naturally extends to an almost complex structure, still denoted by $J$, on the symplectization $(\mathbb{R} \times J^1(M), 
d(e^t \alpha))$ by $J\frac{\partial}{\partial t} = R_\alpha$. We consider the moduli space $\widetilde{\mathcal{M}}(a;b_1, \ldots, b_k)$
of $J$-holomorphic disks in $\mathbb{R}\times J^1(M)$ with boundary on $\mathbb{R} \times \Lambda$ and with $k+1$ 
punctures on the boundary that are asymptotic at the first puncture to the Reeb chord $a$ at $t = +\infty$ and at the other 
punctures to the Reeb chords $b_1, \ldots, b_k$ at $t=-\infty$. For a generic choice of $J$, this moduli space is a smooth 
manifold of dimension $|a|-\sum_{i=1}^k|b_i|$, see~\cite[Proposition 2.2]{EES1}. This moduli space carries a natural 
$\mathbb{R}$-action corresponding to the translation of $J$-holomorphic disks along the $t$ coordinate. When $\{ b_1, \ldots, b_k\} 
\neq \{ a \}$, let us denote by $\mathcal{M}(a;b_1, \ldots, b_k)$ the quotient of this moduli space by this free action.

We define a differential $\partial$ on $\mathcal{A}(\Lambda)$ by
$$
\partial a = \sum_{\stackrel{b_1, \ldots, b_k}{\dim \mathcal{M}(a;b_1, \ldots, b_k)=0}} \#_2 \mathcal{M}(a;b_1,\ldots,b_k) \ b_1 \ldots b_k 
$$
where $\#_2 \mathcal{M}(a;b_1,\ldots,b_k)$ is the number of elements in the corresponding moduli space, modulo $2$.
This differential has degree $-1$ and satisfies $\partial \circ \partial = 0$.

The resulting differential graded algebra $(\mathcal{A}(\Lambda), \partial)$ is called the Chekanov-Eliashberg DGA and its homology
is called Legendrian contact homology and denoted by $LCH(\Lambda)$. This graded algebra over $\Z_2$ depends only on the
Legendrian isotopy class of $\Lambda$. 

Let us now turn to the definition of a linearized version of Legendrian contact homology.
An augmentation of $(\mathcal{A}(\Lambda), \partial)$ is a unital DGA map $\varepsilon : (\mathcal{A}(\Lambda), \partial) \to (\Z_2, 0)$.
In other words, it is a choice of $\varepsilon(c) \in \Z_2$ for all Reeb chords $c$ of $\Lambda$, it satisfies $\varepsilon(1)=1$, it 
extends to $\mathcal{A}(\Lambda)$ multiplicatively and additively, and it satisfies $\varepsilon \circ \partial = 0$. 

Such an augmentation can be used to define a linearization of $(\mathcal{A}(\Lambda), \partial)$. Let $C(\Lambda)$ be the vector
space over $\Z_2$ freely generated by all Reeb chords of $\Lambda$. We also define the linearized differential $\partial^\varepsilon$ 
on $C(\Lambda)$ by
$$
\partial^\varepsilon a = \!\!\!\!\!\! \sum_{\stackrel{b_1, \ldots, b_k}{\dim \mathcal{M}(a;b_1, \ldots, b_k)=0}} \!\!\!\!\!\! \#_2 
\mathcal{M}(a;b_1,\ldots,b_k) 
\sum_{i=1}^k  \varepsilon(b_1) \ldots \varepsilon(b_{i-1}) b_i \varepsilon(b_{i+1}) \ldots \varepsilon(b_k). 
$$
This differential has degree $-1$ and satisfies $\partial^\varepsilon \circ \partial^\varepsilon = 0$. 
The homology of the resulting linearized
complex  $(C(\Lambda),\partial^\varepsilon)$ is called linearized Legendrian contact homology (with respect to $\varepsilon$) and 
denoted by $LCH^\varepsilon(\Lambda)$. The collection of these graded modules over $\Z_2$ for all augmentations of $\Lambda$
depends only on the Legendrian isotopy class of $\Lambda$.

Linearized Legendrian contact homology fits into a duality long exact sequence~\cite{EESdual} together with its cohomological version
$LCH_{\varepsilon}(\Lambda)$ and with the singular homology $H(\Lambda)$ of the underlying $n$-dimensional manifold $\Lambda$:
$$
\ldots \to LCH^{n-k-1}_{\varepsilon}(\Lambda) \to LCH^\varepsilon_k(\Lambda) \stackrel{\tau_k}{\longrightarrow} H_k(\Lambda) 
\to LCH^{n-k}_{\varepsilon}(\Lambda) \to \ldots
$$
Moreover, the map $\tau_n$ in the above exact sequence does not vanish. These properties induce constraints on the graded modules
over $\Z_2$ that can be realized as the linearized Legendrian contact homology of some Legendrian submanifold, with respect to some 
augmentation. These constraints can be formulated in terms of the Poincar\'e polynomial of $LCH^{\varepsilon}(\Lambda)$, which 
is the Laurent polynomial defined by
$$
P_{\Lambda, \varepsilon}(t) = \sum_{k \in \Z} \dim_{\Z_2} LCH^{\varepsilon}_k(\Lambda) \ t^k.
$$
When $\Lambda$ is connected, the duality exact sequence and the non-vanishing of $\tau_n$ imply that the above Poincar\'e
polynomial has the form
\begin{equation} \label{eq:geoglin}
P_{\Lambda, \varepsilon}(t) = q(t) + p(t) + t^{n-1} p(t^{-1}),
\end{equation}
where $q$ is a monic polynomial of degree $n$ with integral coefficients (corresponding to the image of the maps $\tau_k$) 
and $p$ is a Laurent polynomial with integral coefficients (corresponding to the kernel of the maps $\tau_k$).
We shall say that a Laurent polynomial of this form is LCH-admissible.

The first author together with Sabloff and Traynor~\cite{BST} studied generating family homology $GH(f)$, an invariant for isotopy
classes of Legendrian submanifolds $\Lambda \subset (J^1(M),\xi)$ admitting a generating family $f$. This invariant is also a graded
module over $\Z_2$ and satisfies the same duality exact sequence as above. In this study, the effect of Legendrian ambient surgeries 
on this invariant was determined and these operations were used to produce many interesting examples of Legendrian submanifolds
admitting generating families. More precisely, for any LCH-admissible Laurent polynomial $P$, a connected Legendrian submanifold 
$\Lambda_P$ admitting a generating family $f_P$ realizing $P$ as the Poincar\'e polynomial of $GH(f_P)$ was constructed 
using these operations. On the other hand, Dimitroglou Rizell~\cite{DR} showed in particular that Legendrian ambient surgeries
have the same effect as above on linearized Legendrian contact homology (for more details in the case of the connected sum, see 
the proof of Proposition~\ref{prop:connsum}). This result can be used step by step in the constructions of~\cite{BST} to show that, for 
any LCH-admissible Laurent polynomial $P$, there exists an augmentation $\varepsilon_P$ for $\Lambda_P$ such that 
$LCH^{\varepsilon_P}(\Lambda_P) \cong GH(f_P)$. 
Therefore, the geography question for linearized Legendrian contact homology is completely determined by the above LCH-admissible 
Laurent polynomials.

Finally, we turn to a generalization of linearized LCH introduced by the first author together with 
Chantraine~\cite{BC}. Using two augmentations $\varepsilon_1$ and $\varepsilon_2$ of $(\mathcal{A}(\Lambda), \partial)$, 
we can define another differential $\partial^{\varepsilon_1,\varepsilon_2}$ on $C(\Lambda)$, called bilinearized differential:
$$
\partial^{\varepsilon_1,\varepsilon_2} a = \!\!\!\!\!\!\!\!\!\!\!\!\! \sum_{\stackrel{b_1, \ldots, b_k}{\dim \mathcal{M}(a;b_1, \ldots, b_k)=0}} 
\!\!\!\!\!\!\!\!\!\!\!\!\!\!\! \#_2 \mathcal{M}(a;b_1,\ldots,b_k) 
\sum_{i=1}^k  \varepsilon_1(b_1) \ldots \varepsilon_1(b_{i-1}) b_i \varepsilon_2(b_{i+1}) \ldots \varepsilon_2(b_k). 
$$
As above, this differential has degree $-1$ and satisfies $\partial^{\varepsilon_1,\varepsilon_2} \circ 
\partial^{\varepsilon_1,\varepsilon_2} = 0$. The homology of the resulting bilinearized
complex  $(C(\Lambda),\partial^{\varepsilon_1,\varepsilon_2})$ is called bilinearized Legendrian contact homology 
(with respect to $\varepsilon_1$ and $\varepsilon_2$) and denoted by $LCH^{\varepsilon_1,\varepsilon_2}(\Lambda)$. 
The collection of these graded modules over $\Z_2$ for all pairs of augmentations of $\Lambda$ depends only on the 
Legendrian isotopy class of $\Lambda$.

Bilinearized Legendrian contact homology also satisfies a duality exact sequence~\cite{BC}, but one has to take care of
the ordering of the augmentations:
\begin{equation} \label{eq:bilindualityseq}
\ldots \to LCH^{n-k-1}_{\varepsilon_2, \varepsilon_1}(\Lambda) \to LCH^{\varepsilon_1,\varepsilon_2}_k(\Lambda) 
\stackrel{\tau_k}{\longrightarrow} H_k(\Lambda) \stackrel{\sigma_{n-k}}{\longrightarrow} LCH^{n-k}_{\varepsilon_2, \varepsilon_1}
(\Lambda) \to \ldots
\end{equation}
Moreover, unlike in the linearized case, there exist~\cite[section~5]{BC} connected Legendrian submanifolds $\Lambda$ 
with augmentations $\varepsilon_1$ and $\varepsilon_2$ such that the map $\tau_n$ vanishes. Our goal in this article is 
to understand when the map $\tau_n$ vanishes or not, and to study the geography of the Poincar\'e polynomials
$$
P_{\Lambda, \varepsilon_1, \varepsilon_2}(t) = \sum_{k \in \Z} \dim_{\Z_2} 
LCH^{\varepsilon_1,\varepsilon_2}_k(\Lambda) \ t^k.
$$
for bilinearized Legendrian contact homology.

\section{Fundamental classes in bilinearized Legendrian contact homology}
\label{sec:fundclass}

There are several notions of equivalence for augmentations of DGAs that were introduced in the literature and used
in the context of the Chekanov-Eliashberg DGA. As the results of this section will show, it turns out that the equivalence 
relation among augmentations that controls best the behavior of bilinearized LCH is the notion of DGA homotopic 
augmentations~\cite[Definition 5.13]{NRSSZ}. Let $\varepsilon_1, \varepsilon_2$ be two augmentations of the DGA 
$(\mathcal{A}, \partial)$ over $\Z_2$. 
Recall that a linear map $K : \mathcal{A} \to \Z_2$ is said to be an $(\varepsilon_1, \varepsilon_2)$-derivation if $K(ab) = 
\varepsilon_1(a) K(b) + K(a) \varepsilon_2(b)$ for any $a, b \in \mathcal{A}$. We say that $\varepsilon_1$ is DGA 
homotopic to $\varepsilon_2$, and we write $\varepsilon_1 \sim \varepsilon_2$, if there exists an 
$(\varepsilon_1, \varepsilon_2)$-derivation $K : \mathcal{A} \to \Z_2$ of degree
$+1$ such that $\varepsilon_1 - \varepsilon_2 = K \circ \partial$. It is a standard fact that DGA homotopy is an equivalence 
relation~\cite[Lemma 26.3]{FHT}.

Note that the defining condition for a DGA homotopy admits a beautiful and convenient reformulation in terms of the
bilinearized complex.

\begin{Lem}  \label{lem:htpy}
Two augmentations $\varepsilon_1, \varepsilon_2$ are DGA homotopic if and only if there exists a linear map 
$\overline{K} : C(\Lambda) \to \Z_2$ of degree $+1$ such that $\varepsilon_1 - \varepsilon_2 = \overline{K} \circ 
\partial^{\varepsilon_1,\varepsilon_2}$ on $C(\Lambda)$.
\end{Lem}

\begin{proof}
Suppose first that $\varepsilon_1$ is DGA homotopic to $\varepsilon_2$. This implies in particular that
$\varepsilon_1(c) - \varepsilon_2(c) = K \circ \partial c$ for any $c \in C(\Lambda)$. Since $K$ is an 
$(\varepsilon_1, \varepsilon_2)$-derivation, it directly follows from the definition of the bilinearized differential 
that $K \circ \partial c = K \circ \partial^{\varepsilon_1,\varepsilon_2} c$. It then suffices to take $\overline{K}$
to be the restriction of $K$ to $C(\Lambda)$.

Suppose now that there exists a linear map $\overline{K} : C(\Lambda) \to \Z_2$ of degree $+1$ such that 
$\varepsilon_1 - \varepsilon_2 = \overline{K} \circ \partial^{\varepsilon_1,\varepsilon_2}$ on $C(\Lambda)$.
The map $\overline{K}$ determines a unique $(\varepsilon_1, \varepsilon_2)$-derivation $K : \mathcal{A} \to \Z_2$
via the relation $K(a_1 \ldots a_n) = \sum_{i=1}^k \varepsilon_1(a_1 \ldots a_{i-1}) \overline{K}(a_i) 
\varepsilon_2(a_{i+1}\ldots a_n)$ for all $a_1, \ldots, a_n \in \mathcal{A}$.  As above, these maps satisfy
$K \circ \partial c = \overline{K} \circ \partial^{\varepsilon_1,\varepsilon_2} c$, so that $\varepsilon_1 - \varepsilon_2 
= \overline{K} \circ \partial$ on $C(\Lambda)$. Now observe that 
$\varepsilon_1(ab) -\varepsilon_2(ab) = \varepsilon_1(a) \left(\varepsilon_1(b) - \varepsilon_2(b)\right) 
+ \left( \varepsilon_1(a) - \varepsilon_2(a) \right) \varepsilon_2(b)$ and on the other hand $K \circ \partial (ab) =
\varepsilon_1(\partial a) K(b) + \varepsilon_1(a) K(\partial b) + K(\partial a) \varepsilon_2(b) + K(a) \varepsilon_2(\partial  b)
=  \varepsilon_1(a) K (\partial b) + K (\partial a) \varepsilon_2(b)$. Hence if $a,b$ satisfy the DGA homotopy relation,
then $ab$ satisfies it as well. Since this relation holds on $C(\Lambda)$, it follows that it is also satisfied on $\mathcal{A}$.
\end{proof}

Note that, in the above proof, the extension of the linear map $\overline{K}$ to a unique 
$(\varepsilon_1, \varepsilon_2)$-derivation on $\mathcal{A}$, as well as the extension of 
the homotopy relation from $C(\Lambda)$ to $\mathcal{A}$ were
first established in a more general setup by K\'alm\'an in \cite[Lemma~2.18]{Kalman}.

With this suitable notion of equivalence for augmentations, we can now turn to the study of the fundamental class in bilinearized
LCH, via the maps $\tau_0$ and $\tau_n$ from the duality long exact sequence. The following proposition generalizes Theorem 5.5
from~\cite{EESdual}.

\begin{Prop}  \label{prop:tau}
Let $\varepsilon_1, \varepsilon_2$ be augmentations of the Chekanov-Eliashberg DGA $(\mathcal{A}, \partial)$ of a closed, 
connected $n$-dimensional Legendrian submanifold $\Lambda$ in $(J^1(M), \xi)$. The map $\tau_0 : 
LCH^{\varepsilon_1,\varepsilon_2}_0(\Lambda) \to H_0(\Lambda)$ from the duality long exact sequence vanishes if and only
if $\varepsilon_1$ and $\varepsilon_2$ are DGA homotopic.
\end{Prop}

\begin{proof}
Let $f$ be a Morse function on $\Lambda$ with a unique minimum at point $m$ and $g$ be a Riemannian metric on $\Lambda$. 
Since the stable manifold of $m$ is open and dense in $\Lambda$, for a generic choice of the Morse-Smale pair $(f,g)$, the 
endpoints of all Reeb chords of $\Lambda$ are in this stable manifold. 
The vector space $H_0(\Lambda)$ is generated by $m$ and we identify it with $\Z_2$.
By the results of~\cite{EESdual}, the map $\tau_0$ counts rigid $J$-holomorphic disks with boundary on $\Lambda$, with a 
positive puncture on the boundary and with a marked point on the boundary mapping to the stable manifold of $m$. This disk 
can have extra negative punctures on the boundary; these are augmented by $\varepsilon_1$ if they sit between the positive 
puncture and the marked point, and by $\varepsilon_2$ if they sit between the marked point and the positive puncture. Since 
mapping to $m$ is an open condition on $\Lambda$, such rigid configurations can only occur when the image of the disk boundary 
is discrete in $\Lambda$. In other words, the holomorphic disk maps to the symplectization of a Reeb chord $c$ of $\Lambda$. 
Since there is a unique positive puncture, this map is not a covering, and there is a unique negative puncture at $c$. There is a 
unique such $J$-holomorphic disk for any chord $c$ of $\Lambda$. The marked point maps to the starting point or to the ending 
point of the chord $c$ in $\Lambda$. If the marked point maps
to the starting point of $c$, the negative puncture sits between the positive puncture and the marked point on the boundary of the disk,
which therefore contributes $\varepsilon_2(c)$ to $\tau_0(c)$ at chain level. If the marked point maps to the ending point of $c$, the 
negative puncture sits between the marked point and the positive puncture on the boundary of the disk, which therefore contributes 
$\varepsilon_1(c)$ to $\tau_0(c)$. We conclude that the map $\tau_0$ is given at chain level by $\varepsilon_1 - \varepsilon_2$.

If $\varepsilon_1$ and $\varepsilon_2$ are DGA homotopic, then by Lemma~\ref{lem:htpy} the map $\tau_0$ is null homotopic
and therefore vanishes in homology. On the other hand, if $\varepsilon_1$ and $\varepsilon_2$ are not DGA homotopic, 
Lemma~\ref{lem:htpy} implies that the map $\varepsilon_1 - \varepsilon_2 : C_0(\Lambda)\to \Z_2$ does not factor through 
the bilinearized differential $\partial^{\varepsilon_1,\varepsilon_2}$. In other words, there exists $a \in C_0(\Lambda)$ such that
$\partial^{\varepsilon_1,\varepsilon_2} a = 0$ but $\varepsilon_1(a) - \varepsilon_2(a) \neq 0$. But then the homology class
$[a] \in LCH^{\varepsilon_1,\varepsilon_2}_0(\Lambda)$ satisfies $\tau_0([a])\neq 0$, so that $\tau_0$ does not vanish in homology.
\end{proof}

We are now in position to prove the first main result of this paper.

\begin{proof}[Proof of Theorem~\ref{thm:A}]
In the duality long exact sequence~\eqref{eq:bilindualityseq} for bilinearized LCH, the maps $\tau_k$ and $\sigma_k$ are adjoint
in the sense of~\cite[Proposition 3.9]{EESdual} as in the linearized case. The proof of this fact is essentially identical in the 
bilinearized case: the holomorphic disks counted by $\tau_k$ are still in bijective correspondence with those counted by $\sigma_k$.
In the bilinearized case, it is also necessary to use the fact that the extra negative punctures on corresponding disks are augmented 
with the same augmentations, in order to reach the conclusion.

In particular, $\tau_n$ vanishes if and only if $\sigma_n$ vanishes. Since $H_0(\Lambda) \cong \Z_2$, the exactness of the duality
sequence~\eqref{eq:bilindualityseq} implies that $\sigma_n$ vanishes if and only if $\tau_0$ does not vanish. 
By Proposition~\ref{prop:tau}, this means that $\tau_n$ vanishes if and only if the augmentations $\varepsilon_1$ and 
$\varepsilon_2$ are not DGA homotopic.
\end{proof}

This difference in the behavior of bilinearized LCH can be used to determine DGA homotopy classes of augmentations. More
precisely, the next proposition shows that bilinearized LCH provides an explicit criterion to decide whether two augmentations 
are DGA homotopic or not.

\begin{Prop}  \label{prop:aughtpy}
Let $\varepsilon_1, \varepsilon_2$ be augmentations of the Chekanov-Eliashberg DGA $(\mathcal{A}, \partial)$ of a closed, 
connected $n$-dimensional Legendrian submanifold $\Lambda$ in $(J^1(M), \xi)$. Then
$$
\dim_{\Z_2} LCH^{\varepsilon_2,\varepsilon_1}_n(\Lambda) - \dim_{\Z_2} LCH^{\varepsilon_1,\varepsilon_2}_{-1}(\Lambda)
= \left\{ \begin{array}{ll}
0 & \textrm{if } \varepsilon_1 \not\sim \varepsilon_2, \\
1 & \textrm{if } \varepsilon_1 \sim \varepsilon_2.
\end{array}\right.
$$
\end{Prop}

\begin{proof}
By the duality exact sequence~\eqref{eq:bilindualityseq}, we have
$$
H_0(\Lambda) \cong \Z_2 \stackrel{\sigma_n}{\longrightarrow} LCH^n_{\varepsilon_2, \varepsilon_1}(\Lambda) \to 
LCH^{\varepsilon_1, \varepsilon_2}_{-1}(\Lambda) \to H_{-1}(\Lambda) = 0 .
$$
In other words, $LCH^n_{\varepsilon_2, \varepsilon_1}(\Lambda)/\textrm{Im } \sigma_n \cong 
LCH^{\varepsilon_1, \varepsilon_2}_{-1}(\Lambda)$. Taking into account that $\dim_{\Z_2} 
LCH^n_{\varepsilon_2, \varepsilon_1}(\Lambda) = \dim_{\Z_2} LCH_n^{\varepsilon_2, \varepsilon_1}(\Lambda)$, 
we obtain the desired result since, as in the proof of Theorem~\ref{thm:A}, the rank of $\sigma_n$ is $1$ when
$\varepsilon_1 \sim \varepsilon_2$ and vanishes when $\varepsilon_1 \not\sim \varepsilon_2$. 
\end{proof}

Corollary~\ref{cor:A} follows immediately from the above proposition.

\begin{Ex} \label{ex:m8_21}
Let us consider the Legendrian knot $K_2$ studied by Melvin and Shrestha in~\cite[Section~3]{MS}, which is topologically
the mirror image of the knot $8_{21}$, and illustrated in Figure~\ref{fig:m8_21}.

\begin{figure}
\labellist
\small\hair 2pt
\endlabellist
  \centerline{\includegraphics[width=6cm, height=3cm]{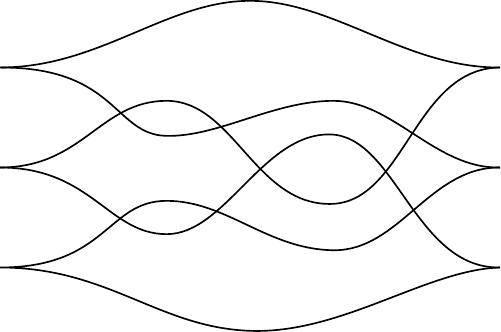}}
  \caption{Front projection of the Legendrian knot $K_2$.}
  \label{fig:m8_21}
\end{figure}

It is shown in~\cite[Section~3]{MS} that the Chekanov-Eliashberg DGA of this Legendrian knot $K_2$ has exactly $16$ augmentations,
which split into a set $A$ of $4$ augmentations and a set $B$ of $12$ augmentations such that
$P_{K_2, \varepsilon}(t) = 2t+4 + t^{-1}$ if $\varepsilon \in A$ and $P_{K_2, \varepsilon}(t) = t + 2$ if 
$\varepsilon \in B$. This implies that augmentations in $A$ are not DGA homotopic to augmentations 
in $B$. However, the number of DGA homotopy classes of augmentations for $K_2$ was not determined 
in~\cite{MS}, as linearized LCH does not suffice to obtain this information.

Using Proposition~\ref{prop:aughtpy}, these DGA homotopy classes can be determined systematically.
It turns out that the augmentations in $A$ are pairwise not DGA homotopic, because the 
Poincar\'e polynomial of any such pair of augmentations  is $t + 3 + t^{-1}$.
On the other hand,
the set $B$ splits into $6$ DGA homotopy classes $\mathcal{C}_1, \ldots, \mathcal{C}_6$ 
of augmentations. The bLCH Poincar\'e polynomials are given by $t+2$ for two DGA homotopic augmentations 
in $B$, by $1$ for two non DGA homotopic  augmentations both in $\mathcal{C}_1 
\cup \mathcal{C}_2 \cup \mathcal{C}_3$ or in $\mathcal{C}_4 \cup \mathcal{C}_5
\cup \mathcal{C}_6$, and by $t+2$ and $2 + t^{-1}$ otherwise.

These calculations are straightforward but tedious. A suitable Python code run by a computer gives
the above answer instantly.
\end{Ex}

We conclude our study of the fundamental classes in bilinearized LCH with a useful description 
of their behavior when performing a connected sum. To this end, it is convenient to introduce some additional notation
about the map $\tau_n$ in the duality exact sequence~\eqref{eq:bilindualityseq}. Its target space $H_n(\Lambda)$ is 
spanned by the fundamental classes $[\Lambda_i]$ of the connected components $\Lambda_i$ of the Legendrian submanifold
$\Lambda$. We can therefore decompose $\tau_n$ as $\sum_i \tau_{n,i} [\Lambda_i]$, where the maps $\tau_{n,i}$ take
their values in $\Z_2$.

\begin{Prop}  \label{prop:connsum}
Let $\Lambda$ be a Legendrian link in $J^1(M)$ equipped with two augmentations $\varepsilon_1$ 
and $\varepsilon_2$. Let $\overline{\Lambda}$ be the Legendrian submanifold obtained by performing
a connected sum between two connected components $\Lambda_1$ and $\Lambda_2$ of $\Lambda$,
and let $\overline{\varepsilon}_1$ and $\overline{\varepsilon}_2$ be the augmentations induced by
$\varepsilon_1$ and $\varepsilon_2$.

If the map $\tau_{n,1} - \tau_{n,2}$ constructed from the map $\tau_n$
in the duality exact sequence~\eqref{eq:bilindualityseq} vanishes, then 
$P_{\overline{\Lambda}, \overline{\varepsilon}_1, \overline{\varepsilon}_2}(t) 
= P_{\Lambda, \varepsilon_1, \varepsilon_2}(t) + t^{n-1}$. Otherwise,
$P_{\overline{\Lambda}, \overline{\varepsilon}_1, \overline{\varepsilon}_2}(t) 
= P_{\Lambda, \varepsilon_1, \varepsilon_2}(t) - t^n$.
\end{Prop}

\begin{proof}
As explained in~\cite[Section~3.2.5]{BC}, the map $\tau_n$ in the duality exact sequence~\eqref{eq:bilindualityseq} for $\Lambda$ counts 
holomorphic disks in the symplectization of $J^1(M)$ with boundary on the symplectization of $\Lambda$, having a positive puncture asymptotic 
to a chord $c$ of $\Lambda$ and a marked point on the boundary mapped to a fixed generic point $p_j$ of a connected component 
$\Lambda_j$ of $\Lambda$. This disk can also carry negative punctures on the boundary; let us say that those located between the positive 
puncture and the chord (with respect to the natural orientation of the boundary) are asymptotic to chords $c^-_1, \ldots, c^-_r$, while those
between the marked point and the positive puncture are asymptotic to $c^-_{r+1}, \ldots, c^-_{r+s}$. Let us denote by 
$\mathcal{M}(c;c^-_1, \ldots, c^-_r, p_j, c^-_{r+1}, \ldots, c^-_{r+s})$ the moduli space of such holomorphic disks, modulo translation in 
the $\R$ direction of the symplectization. The map $\tau_n$ is then given by
\begin{eqnarray*}
\lefteqn{\tau_n(c) = \sum_j \#_2 \mathcal{M}(c;c^-_1, \ldots, c^-_r, p_j, c^-_{r+1}, \ldots, c^-_{r+s})} \\
&& \hspace{2cm} \varepsilon_1(c^-_1) \ldots \varepsilon_1(c^-_r) \varepsilon_2(c^-_{r+1}) \ldots \varepsilon_2(c^-_{r+s}) [\Lambda_j].
\end{eqnarray*}

On the other hand, the effect of a connected sum on bilinearized LCH can be deduced from the results of Dimitroglou Rizell on the full 
Chekanov-Eliashberg DGA~\cite[Theorem~1.6]{DR}. There is an isomorphism of DGAs $\Psi : (\mathcal{A}(\overline{\Lambda}), \partial_{\overline{\Lambda}}) \to (\mathcal{A}(\Lambda; S), \partial_S)$ between the Chekanov-Eliashberg DGA of $\overline{\Lambda}$ 
and the DGA $(\mathcal{A}(\Lambda; S), \partial_S)$ generated by the Reeb chords of $\Lambda$ as well as a formal generator $s$ of 
degree $n-1$, equipped with a differential $\partial_S$ satisfying in particular $\partial_S s = 0$. In this notation, $S$ stands for the pair of points $\{ p_1
\in \Lambda_1, p_2 \in \Lambda_2 \}$ in a neighborhood of which the connected sum is performed. Any augmentation $\varepsilon$ of the 
Chekanov-Eliashberg DGA of $\Lambda$ can be extended to an augmentation of $(\mathcal{A}(\Lambda; S), \partial_S)$ by setting 
$\varepsilon(s) = 0$. Moreover, the pullback $\Psi^*\varepsilon$ of this extension to the Chekanov-Eliashberg DGA of $\overline{\Lambda}$
coincides with the augmentation induced on $\overline{\Lambda}$ from the original augmentation $\varepsilon$ for $\Lambda$ via the
surgery Lagrangian cobordism between $\overline{\Lambda}$ and $\Lambda$. In particular, we have $\overline{\varepsilon}_1 = 
\Psi^* \varepsilon_1$ and $\overline{\varepsilon}_2 = \Psi^* \varepsilon_2$. Applying the bilinearization procedure to the map $\Psi$, 
we obtain a chain complex isomorphism $\Psi^{\varepsilon_1, \varepsilon_2}$ between the bilinearized chain complex for $\overline{\Lambda}$
and the chain complex $(C(\Lambda, S), \partial_S^{\varepsilon_1, \varepsilon_2})$ generated by Reeb chords of $\Lambda$ and the formal 
generator $s$. Since $\partial_S^{\varepsilon_1, \varepsilon_2} s = 0$, the line spanned by $s$ forms a subcomplex of $(C(\Lambda, S), \partial_S^{\varepsilon_1, \varepsilon_2})$. Moreover, the quotient complex is exactly the bilinearized chain complex for $\Lambda$. 
We therefore obtain a long exact sequence in homology
$$
\ldots \to LCH^{\overline{\varepsilon}_1, \overline{\varepsilon}_2}_k(\overline{\Lambda}) \to LCH^{\varepsilon_1, \varepsilon_2}_k(\Lambda) 
\stackrel{\rho_k}{\to} \Z_2[s]_{k-1} \to LCH^{\overline{\varepsilon}_1, \overline{\varepsilon}_2}_{k-1}(\overline{\Lambda}) \to \ldots
$$
that corresponds to the long exact sequence obtained in~\cite[Theorem~2.1]{BST}
for generating family homology. This exact sequence implies that bilinearized LCH remains unchanged by a connected sum, except possibly in
degrees $n-1$ and $n$. The map $\rho_n$ is the part of the bilinearized differential $\partial_S^{\varepsilon_1, \varepsilon_2}$ from the bilinearized
complex for $\Lambda$ to the line spanned by $s$. According to the definition~\cite[Section~1.1.3]{DR} of $\partial_S$ and the above description of 
$\tau_n$, this map is given by $\rho_n = (\tau_{n,1} - \tau_{n,2}) s$. 

If $\rho_n=0$, the generator $s$ injects into $LCH_{n-1}^{\overline{\varepsilon}_1, \overline{\varepsilon}_2}(\overline{\Lambda})$, resulting
in an exact term $t^{n-1}$ in the Poincar\'e polynomial.  If $\rho_n \neq 0$, the map $LCH^{\overline{\varepsilon}_1, \overline{\varepsilon}_2}_n(\overline{\Lambda}) \to LCH^{\varepsilon_1, \varepsilon_2}_n(\Lambda)$ has a $1$-dimensional cokernel, resulting in the loss of a term
$t^n$ in the Poincar\'e polynomial.
\end{proof}

\section{Geography of bilinearized Legendrian contact homology}
\label{sec:geog}

In this section, we study the possible values for the Poincar\'e polynomial $P_{\Lambda, \varepsilon_1, \varepsilon_2}$
of the bilinearized LCH for a closed, connected Legendrian submanifold $\Lambda$ in $J^1(M)$ with $\dim M = n$,
equipped with two augmentations $\varepsilon_1$ and $\varepsilon_2$ of its Chekanov-Eliashberg DGA.

When $\varepsilon_1 = \varepsilon_2$, this geography question was completely answered in~\cite{BST} for generating
family homology. As explained in Section~\ref{sec:bLCH}, this result extends to linearized LCH via the work of 
Dimitroglou Rizell~\cite{DR}. Moreover, bilinearized LCH is invariant under changes of augmentations within their
DGA homotopy classes~\cite[Section~5.3]{NRSSZ}. Therefore, the geography of bilinearized LCH is already known 
when $\varepsilon_1 \sim \varepsilon_2$.

%
%
\subsection{Basic properties of bLCH Poincar\'e polynomials}

We now turn to the case $\varepsilon_1 \not\sim \varepsilon_2$, and describe the possible Poincar\'e polynomials
for bilinearized LCH.

\begin{defn}  \label{def:bLCHadm}
A bLCH-admissible polynomial is the data of a Laurent polynomial $P$ with nonnegative integral coefficients together
with a splitting $P = q + p$ invoving two Laurent polynomials with nonnegative integral coefficients $p$ and $q$ such that
\begin{enumerate}
\item[(i)] $q$ is a polynomial of degree at most $n-1$ with $q(0)=1$,
\item[(ii)] $p(-1)$ is even if $n=1$ and $p(-1) \le \frac12(1-q(-1))$ if $n =2$.
\end{enumerate}
\end{defn}

We first show that the Poincar\'e polynomial of bilinearized LCH always has this form.

\begin{Prop}  \label{prop:bLCHadm}
Let $\varepsilon_1, \varepsilon_2$ be augmentations of the Chekanov-Eliashberg DGA $(\mathcal{A}, \partial)$ of a closed, 
connected $n$-dimensional Legendrian submanifold $\Lambda$ with vanishing Maslov class in $(J^1(M), \xi)$. 
If $\varepsilon_1$ and  $\varepsilon_2$ are not
DGA homotopic, then the Poincar\'e polynomial $P_{\Lambda, \varepsilon_1, \varepsilon_2}$ corresponding to 
$LCH^{\varepsilon_1, \varepsilon_2}(\Lambda)$ is bLCH-admissible.
\end{Prop}

\begin{proof}
Considering the map $\tau_k$ from the duality exact sequence~\eqref{eq:bilindualityseq}, we 
have the relation $\dim_{\Z_2} LCH^{\varepsilon_1, \varepsilon_2}_k(\Lambda) = \dim_{\Z_2} \ker \tau_k + 
\dim_{\Z_2} \textrm{im } \tau_k$. Let $p$ and $q$ be the Poincar\'e polynomials constructed using the terms
in the right hand side of this relation: $p(t) = \sum_{k \in \Z} \dim_{\Z_2} \ker \tau_k \, t^k$ and
$q(t) = \sum_{k \in \Z} \dim_{\Z_2} \textrm{im } \tau_k \, t^k$. This provides the desired splitting
$P_{\Lambda, \varepsilon_1, \varepsilon_2} = q + p$.

Let us prove (i).
Since $\textrm{im } \tau_k \subset H_k(\Lambda)$, $q$ is a polynomial of degree at most $n$. By Proposition~\ref{prop:tau}, since
$\varepsilon_1 \not\sim \varepsilon_2$, $\textrm{im } \tau_0 \neq 0$. But $H_0(\Lambda) = \Z_2$ as $\Lambda$ is connected, so
that $q(0)=1$. On the other hand, by Theorem~\ref{thm:A}, since $\varepsilon_1 \not\sim \varepsilon_2$, $\tau_n=0$ so that the 
term of degree $n$ in $q$ vanishes and $q$ is a polynomial of degree at most $n-1$.

Let us now prove (ii). Assume first that $n$ is odd. 
Since the generators of the chain complex $C(\Lambda)$ do not depend on the augmentations, 
the Euler characteristic $P_{\Lambda, \varepsilon_1, \varepsilon_2}(-1)$ does not depend on 
the augmentations either. Equation~\eqref{eq:geoglin}
then implies that $P_{\Lambda, \varepsilon_1, \varepsilon_2}(-1)$ has the same parity as
$\frac12 \sum_{k \in \Z} \dim_{\Z_2} H_k(\Lambda)$, since $(-1)^{n-1} =1$ when $n$ is odd. 
If $n=1$, then condition (i) sets $q(t)=1$ so that $q(-1)=1$ while 
$\frac12 \sum_{k \in \Z} \dim_{\Z_2} H_k(\Lambda) = 1$. By subtraction, we deduce that $p(-1)$
must be even. Note that if $n \ge 3$, this does not impose any condition on $p(-1)$ since $q(-1)$ 
can take arbitrary integer values.

Assume now that $n$ is even. 
By~\cite[Proposition~3.3]{EES2}, the Thurston-Bennequin invariant of $\Lambda$ is given by $tb(\Lambda) = 
(-1)^{\frac{(n-1)(n-2)}2} P_{\Lambda, \varepsilon_1, \varepsilon_2}(-1)$. On the other hand, by~\cite[Proposition~3.2]{EES2}, 
$tb(\Lambda) = (-1)^{\frac{n}2 +1} \frac12 \mathcal{X}(\Lambda)$ when $n$ is even. Hence 
$P_{\Lambda, \varepsilon_1, \varepsilon_2}(-1) = \frac12 \mathcal{X}(\Lambda)$. When $n=2$, 
$\frac12 \mathcal{X}(\Lambda) = \frac12 (1 - \dim_{\Z_2} H_1(\Lambda) + 1) \le \frac12 (1 + q(-1))$. 
By subtraction, we get that $p(-1) \le \frac12(1 - q(-1))$.
Note that if $n \ge 4$, this does not impose any condition on $p(-1)$ since $\frac12 \mathcal{X}(\Lambda)$
can take arbitrary integer values.
\end{proof}

\begin{Rem}  \label{rmk:admsphere}
If we restrict ourselves to Legendrian spheres $\Lambda$, the Laurent polynomials $P= q+ p$ that can arise as
the Poincar\'e polynomial of bilinearized LCH can also be characterized. 
More precisely, revisiting the proof of Proposition~\ref{prop:bLCHadm} shows that in this case the 
polynomials $q$ and $p$ satisfy the more restrictive conditions
\begin{enumerate}
\item[(i')] $q(t) = 1$,
\item[(ii')] $p(-1)$ is even if $n$ is odd and $p(-1) = 0$ if $n$ is even.
\end{enumerate}
\end{Rem}

The duality exact sequence imposes less restrictions on $LCH^{\varepsilon_1, \varepsilon_2}(\Lambda)$ than in the case of 
linearized LCH because it mainly relates this invariant to $LCH^{\varepsilon_2, \varepsilon_1}(\Lambda)$ with exchanged 
augmentations. This fact, however, means that one of these invariants determines the other one. In order to formulate this
more precisely, let us consider the duality exact sequence obtained from~\eqref{eq:bilindualityseq} after reversing the ordering
of the augmentations:
\begin{equation} \label{eq:bilindualityseq2}
\ldots \to LCH^{n-k-1}_{\varepsilon_1, \varepsilon_2}(\Lambda) \to LCH^{\varepsilon_2,\varepsilon_1}_k(\Lambda) 
\stackrel{\tilde{\tau}_k}{\longrightarrow} H_k(\Lambda) \stackrel{\tilde{\sigma}_{n-k}}{\longrightarrow} 
LCH^{n-k}_{\varepsilon_1, \varepsilon_2}(\Lambda) \to \ldots
\end{equation}
In the next Proposition, we denote by $P_{\Lambda}(t)$ the Poincar\'e polynomial for the singular homology of $\Lambda$
with coefficients in $\Z_2$.

\begin{Prop}  \label{prop:relpoly}
Let $\varepsilon_1, \varepsilon_2$ be non DGA homotopic augmentations of the Chekanov-Eliashberg DGA $(\mathcal{A}, \partial)$ 
of a closed, connected $n$-dimensional Legendrian submanifold $\Lambda$ with vanishing Maslov class in $(J^1(M), \xi)$. 
If $P_{\Lambda, \varepsilon_1, \varepsilon_2}$ decomposes as $q + p$ as in Definition~\ref{def:bLCHadm}, then  
$P_{\Lambda, \varepsilon_2, \varepsilon_1}$ decomposes as $\tilde{q} + \tilde{p}$ 
with $\tilde{q}(t) = P_{\Lambda}(t) - t^n q(t^{-1})$ and $\tilde{p}(t) =  t^{n-1} p(t^{-1})$.
\end{Prop}

\begin{proof}
Let us decompose $P_{\Lambda, \varepsilon_2, \varepsilon_1}(t) = \tilde{q}(t) + \tilde{p}(t)$ as in Definition~\ref{def:bLCHadm}.
The polynomial $p$ was defined as $p(t) = \sum_{k \in \Z} \dim_{\Z_2} \ker \tau_k \, t^k$ in the proof of 
Proposition~\ref{prop:bLCHadm}.
But $\ker \tau_k$ is the image of the map $LCH_{\varepsilon_2, \varepsilon_1}^{n-k-1}(\Lambda) \to 
LCH^{\varepsilon_1, \varepsilon_2}_k(\Lambda)$, which is isomorphic to a supplementary subspace 
of $\textrm{im } \sigma_{n-k-1}$ in $LCH_{\varepsilon_2, \varepsilon_1}^{n-k-1}(\Lambda)$. Since $\sigma_{n-k-1}$ 
is the adjoint in the sense of~\cite[Proposition~3.9]{EESdual} of the map 
$\tilde{\tau}_{n-k-1} : LCH^{\varepsilon_2, \varepsilon_1}_{n-k-1}(\Lambda) \to H_{n-k-1}(\Lambda)$, the spaces
$\ker \tau_k$ and $\ker \tilde{\tau}_{n-k-1}$ are isomorphic. Therefore, the polynomial $\tilde{p}$ is given by
$\tilde{p}(t) = \sum_{k \in \Z} \dim_{\Z_2} \ker \tau_k \, t^{n-k-1} = t^{n-1} p(t^{-1})$.

On the other hand, we have $\tilde{q}(t) = \sum_{k \in \Z} \dim_{\Z_2} \textrm{im }\tilde{\tau}_k \, t^k$ as in the proof of 
Proposition~\ref{prop:bLCHadm}. But $\textrm{im }\tilde{\tau}_k = \ker\tilde{\sigma}_{n-k}$ and since $\tau{\sigma}_{n-k}$ 
is the adjoint in the sense of~\cite[Proposition~3.9]{EESdual} of the map $\tau_{n-k}$, we have that $\ker\tilde{\sigma}_{n-k}$
is isomorphic to a supplementary subspace of $\textrm{im }\tau_{n-k}$ in $H_{n-k}(\Lambda)$.
Hence, this means that 
$$
\tilde{q}(t) = \sum_{k \in \Z} \left( \dim_{\Z_2} H_{n-k}(\lambda) -  \dim_{\Z_2} \textrm{im }\tau_{n-k} \right) \, t^k 
= P_{\Lambda}(t) - t^n q(t^{-1})
$$
as announced.  
\end{proof}

Note that, since the data of $P_{\Lambda, \varepsilon_1, \varepsilon_2}$ and $P_{\Lambda, \varepsilon_2, \varepsilon_1}$
determine $P_{\Lambda}$, the question of finding $\Lambda$, $\varepsilon_1$ and $\varepsilon_2$ with prescribed polynomials
$P_{\Lambda, \varepsilon_1, \varepsilon_2}$ and $P_{\Lambda, \varepsilon_2, \varepsilon_1}$ is more complicated than our
geography question. We will not address this more complicated question.

%
%
\subsection{Motivating example}

We now describe a fundamental example in view of the construction of Legendrian submanifolds and augmentations realizing
bLCH-admissible polynomials.

\begin{Ex}  \label{ex:trefoil}
With $n=1$, consider the right handed trefoil knot $\Lambda$ with maximal Thurston-Bennequin invariant, depicted in its front 
projection in Figure~\ref{fig:trefoil}. The same Legendrian knot was already studied in Section 5.1 of~\cite{BC}. We consider it
this time in the front projection, after applying Ng's resolution procedure~\cite{N}.

\begin{figure}
\labellist
\small\hair 2pt
\pinlabel {$b_1$} [bl] at 28 44
\pinlabel {$b_2$} [bl] at 78 44
\pinlabel {$b_3$} [bl] at 133 44
\pinlabel {$a_2$} [bl] at 200 32
\pinlabel {$a_1$} [bl] at 200 56
\endlabellist
  \centerline{\includegraphics[width=6cm, height=3.5cm]{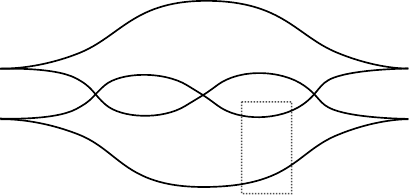}}
  \caption{Front projection of the maximal $tb$ right handed trefoil.}
  \label{fig:trefoil}
\end{figure}

The Chekanov-Eliashberg DGA has 5 generators: $a_1$ and $a_2$ correspond to right cusps and have grading $1$, while
$b_1, b_2$ and $b_3$ correspond to crossings and have grading $0$. The differential is given by
\begin{eqnarray*}
\partial a_1 &=& 1 + b_1 + b_3 + b_1 b_2 b_3, \\
\partial a_2 &=& 1 + b_1 + b_3 + b_3 b_2 b_1.
\end{eqnarray*}
This DGA admits 5 augmentations $\varepsilon_1, \ldots, \varepsilon_5$ given by 
\begin{center}
\begin{tabular}{c|ccc}
& $b_1$ & $b_2$ & $b_3$ \\
\hline
$\varepsilon_1$ & 1 & 1 & 1 \\
$\varepsilon_2$ & 1 & 0 & 0 \\
$\varepsilon_3$ & 1 & 1 & 0 \\
$\varepsilon_4$ & 0 & 0 & 1 \\
$\varepsilon_5$ & 0 & 1 & 1 \\
\end{tabular}
\end{center}

A straightforward calculation shows that $P_{\Lambda, \varepsilon_i, \varepsilon_j}(t) = 1$ for all $i \neq j$. 
In view of Definition~\ref{def:bLCHadm} and Proposition~\ref{prop:bLCHadm}, this is the simplest Poincar\'e polynomial 
that can be obtained using bilinearized LCH.

In order to produce other terms in this Poincar\'e polynomial, let us replace the portion of $\Lambda$ contained in the
dotted rectangle in Figure~\ref{fig:trefoil} by the fragment represented in Figure~\ref{fig:tangle}. This produces a Legendrian
link $\Lambda'$.

\begin{figure}
\labellist
\small\hair 2pt
\pinlabel {$c_1$} [bl] at 48 68
\pinlabel {$c_2$} [bl] at 58 118
\pinlabel {$d_1$} [bl] at 128 68
\pinlabel {$d_2$} [bl] at 118 118
\pinlabel {$a_3$} [bl] at 163 33
\endlabellist
  \centerline{\includegraphics[width=6cm, height=3.5cm]{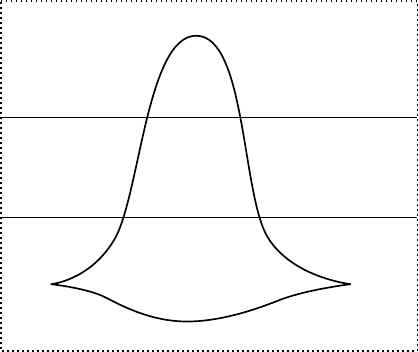}}
  \caption{Replacement for the dotted rectangle in Figure~\ref{fig:trefoil}.}
  \label{fig:tangle}
\end{figure}

The additional generator $a_3$ corresponds to a right cusp and has grading $1$. The 4 mixed chords between
the unknot and the trefoil have a grading that depends on a shift $k \in \Z$ between the Maslov potentials of the
trefoil and of the unknot. These gradings are given by
$$
|c_1| = k-1, \qquad |c_2| = k, \qquad |d_1|=1-k, \qquad |d_2| = -k.
$$
The augmentations $\varepsilon_1, \ldots, \varepsilon_5$ can be extended to this enlarged DGA by sending all new generators 
to $0$. The bilinearized differential of the original generators is therefore unchanged. The differential of the new generators is, 
on the other hand, given by
$$
\partial c_1 = 0, \
\partial c_2 = (1 + b_2 b_1) c_1, \
\partial d_1 = d_2 (1 + b_2 b_1), \
\partial d_2 = 0, \
\partial a_3 = d_1 c_1 + d_2 c_2.
$$
If we choose $\varepsilon_L = \varepsilon_1$ or $\varepsilon_3$ and $\varepsilon_R = \varepsilon_2, \varepsilon_4$ or 
$\varepsilon_5$, then the bilinearized differential is
$$
\partial^{\varepsilon_L, \varepsilon_R} c_1 = 0, \
\partial^{\varepsilon_L, \varepsilon_R} c_2 = 0, \
\partial^{\varepsilon_L, \varepsilon_R} d_1 = d_2, \
\partial^{\varepsilon_L, \varepsilon_R} d_2 = 0, \
\partial^{\varepsilon_L, \varepsilon_R} a_3 = 0.
$$
The Poincar\'e polynomial of the resulting homology is therefore $P_{\Lambda',\varepsilon_L, \varepsilon_R}(t)=t^k + t^{k-1} + t + 1$.
We now perform a connected sum between the right cusps corresponding to $a_2$ and $a_3$ in order to obtain the connected 
Legendrian submanifold $\Lambda''$ represented by Figure~\ref{fig:trefoilnote}. A Legendrian isotopy involving a number of first
Reidemeister moves is performed before the connected sum in order to ensure that the Maslov potentials agree on the cusps to
be merged. This connected sum induces a Lagrangian cobordism $L$ from $\Lambda''$ to $\Lambda'$, and we can use this 
cobordism to pullback the augmentations $\varepsilon_L$ and $\varepsilon_R$ to the Chekanov-Eliashberg DGA of $\Lambda''$.

\begin{figure}
\labellist
\small\hair 2pt
\pinlabel {$\left. \vcenter{\vspace{15mm}} \right\} 1-k$} [bl] at 220 2
\endlabellist
  \centerline{\includegraphics[width=6cm, height=3.5cm]{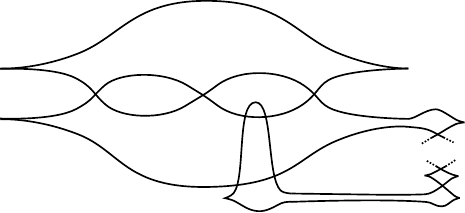}}
  \caption{Front projection of the Legendrian knot $\Lambda''$.}
  \label{fig:trefoilnote}
\end{figure}

By Proposition~\ref{prop:connsum}, since $[a_3] \in LCH_1^{\varepsilon_L, \varepsilon_R}(\Lambda')$ corresponds to the fundamental class 
of the Legendrian unknot depicted in Figure~\ref{fig:trefoilnote}, 
we obtain the Poincar\'e polynomial $P_{\Lambda'',\varepsilon_L, \varepsilon_R}(t)=t^k + t^{k-1} + 1$. 
This corresponds to $q(t) = 1$ and $p(t) = t^k + t^{k-1}$ in Definition~\ref{def:bLCHadm}. 
\end{Ex}

%
%
\subsection{A family of Legendrian spheres with a basic bLCH Poincar\'e polynomial}

In order to generalize Example~\ref{ex:trefoil} to higher dimensions, let us consider the standard Legendrian Hopf
link, or in other words the 2-copy of the standard Legendrian unknot $\Lambda^{(2)}  \subset J^1(\R^n)$. This will lead to a 
generalization of the trefoil knot from Figure~\ref{fig:trefoil}, since the latter can be obtained from the standard Legendrian Hopf 
link in $\R^3$ via a connected sum. Let us denote by $\ell$ the length of the unique Reeb chord of the standard Legendrian unknot
and by $\varepsilon$ the positive shift (much smaller than $\ell$) in the Reeb direction between the two components $\Lambda_1$ and 
$\Lambda_2$ of $\Lambda^{(2)}$. 
We assume that the top component is perturbed by a Morse function of amplitude $\delta$ much smaller than $\varepsilon$ with
exactly one maximum $M$ and one minimum $m$. In particular, among the continuum of Reeb chords of length $\varepsilon$ between 
the two components, only two chords corresponding to these extrema persist after perturbation. We also assume that thanks to this
perturbation, all Reeb chords of $\Lambda^{(2)}$ lie above distinct points of $\R^n$. In order to define the grading of mixed Reeb chords
in this link, we choose the Maslov potential of the upper component $\Lambda_2$ to be given by the Maslov potential of the lower 
component $\Lambda_1$ plus $k$.

\begin{Prop}  \label{prop:hopf}
The Chekanov-Eliashberg DGA of $\Lambda^{(2)} \subset J^1(\R^n)$ has the following 6 generators
\begin{center}
\begin{tabular}{c|cc}
& grading & length \\
\hline
$c_{11}$ & $n$ & $\ell$ \\
$c_{22}$ & $n$ & $\ell$ \\
$c_{12}$ & $n+k$ & $\ell+\varepsilon$ \\
$c_{21}$ & $n-k$ & $\ell-\varepsilon$ \\
$m_{12}$ & $k-1$ & $\varepsilon-\delta$ \\
$M_{12}$  & $n+k-1$ & $\varepsilon+\delta$
\end{tabular}
\end{center}
and its differential is given by
\begin{eqnarray*}
\partial c_{12} &=& M_{12} + m_{12} c_{11} + c_{22} m_{12}, \\
\partial c_{11} &=& c_{21} m_{12}, \\
\partial c_{22} &=& m_{12} c_{21},
\end{eqnarray*}
and $\partial M_{12} = \partial m_{12} = \partial c_{21} = 0$.
\end{Prop}

\begin{proof}
The front projection of each component in $\Lambda^{(2)}$ consists of two sheets, having parallel tangent hyperplanes above a single point
of $\R^n$ before the perturbation by the Morse function. The number of Reeb chords above that point is the number of pairs of sheets, 
which is $\frac{4(4-1)}2 = 6$. The chords between the two highest or the two lowest sheets belong to a continuum of chords of length $\varepsilon$
between the two components, which is replaced by two chords $M_{12}$ for the maximum $M$ and $m_{12}$ for the minimum $m$ after the 
perturbation by the Morse function. Their lengths are therefore $\varepsilon \pm \delta$. Their gradings are given by the Morse index of the 
corresponding critical point plus the difference of Maslov potentials minus one, so that we obtain $n+k-1$ and $k-1$.

The four other chords will be denoted by $c_{ij}$, where $i$ numbers the component of origin for the chord and $j$ numbers the component 
of the endpoint of the chord. Each of these chords corresponds to a maximum of the local difference function between the heights of the sheets
it joins. We therefore obtain the announced gradings and lengths.

The link $\Lambda^{(2)}$ and its Reeb chords determine a quiver represented in Figure~\ref{fig:quiver}, in which each component of the link 
corresponds to a vertex and each Reeb chord corresponds to an oriented edge. When computing the differential of a generator, the terms to be
considered correspond to paths formed by a sequence of edges in this quiver with the same origin and endpoint as the generator, with total grading
one less than the grading of the generator and with total length strictly smaller than the length of the generator.

\begin{figure}
\labellist
\small\hair 2pt
\pinlabel {$c_{11}$} [bl] at -18 10
\pinlabel {$c_{22}$} [bl] at 187 10
\pinlabel {$c_{12}$} [bl] at 87 60
\pinlabel {$M_{12}$} [bl] at 87 43
\pinlabel {$m_{12}$} [bl] at 82 28
\pinlabel {$c_{21}$} [bl] at 87 -12
\pinlabel {$\Lambda_1$} [bl] at 32 -4
\pinlabel {$\Lambda_2$} [bl] at 142 -4
\endlabellist
  \centerline{\includegraphics[width=6cm]{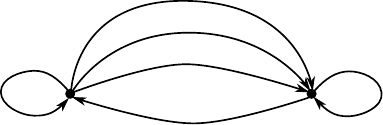}}
  \caption{Quiver corresponding to the standard Hopf link.}
  \label{fig:quiver}
\end{figure}

For $\partial c_{12}$, the only possible terms are $M_{12}$, $m_{12} c_{11}$ and $c_{22} m_{12}$. Indeed, $c_{21}$ cannot appear in such terms 
because two other chords from $\Lambda_1$ to $\Lambda_2$ would be needed as well. The resulting total length would be smaller than the length of 
$c_{12}$ only in the case of $m_{12} c_{21} m_{12}$, but this term is of grading $2$ lower than $c_{12}$. The generators $c_{11}$ and $c_{22}$ can 
appear at most once due to their length, and due to total length constraint, only $m_{12}$ can appear (only once) as a factor, leading to the possibilities 
 $m_{12} c_{11}$ and $c_{22} m_{12}$. Finally, if $M_{12}$ appears, then no other chord can appear as a factor by the previous discussion, 
 leading to the possibility $M_{12}$.

Let us show that each possible term in  $\partial c_{12}$ is realized by exactly one Morse flow tree~\cite{E}, which in turn corresponds to a 
unique holomorphic curve. To obtain $M_{12}$, we start at the chord $c_{12}$ and follow the negative gradient of the local height difference 
function, in the unique direction leading to the chord $M_{12}$. At this chord, we have a $2$-valent puncture of the Morse flow tree and we
continue by following the negative gradient of the local height difference function corresponding to one of the components $\Lambda_1$ or $\Lambda_2$
(depending on which hemisphere the maximum $M$ is located). This gradient trajectory will generically not hit any other Reeb chord and will
finally hit the cusp equator of that component, which is the end of the Morse flow tree. To obtain $m_{12} c_{11}$, we start at the chord $c_{12}$ 
and follow the negative gradient of the local height difference function, in the unique direction leading to the chord $c_{11}$. At this chord, we have 
a $2$-valent puncture of the Morse flow tree and we continue by following the negative gradient of the local height difference function corresponding to
the highest two sheets, which is the Morse function used to perturb the Hopf link. Generically, this gradient trajectory will reach the minimum $m$ so
that we obtain a $1$-valent puncture of the Morse flow tree at $m_{12}$. The term $c_{22} m_{12}$ is obtained similarly.

For $\partial c_{11}$, the only possible term is $c_{21} m_{12}$. Indeed, when $n>1$, the chord $c_{21}$ is the only one available to start an admissible path 
from $\Lambda_1$ to itself, because the empty path is not admissible. When $n=1$, the empty path is admissible but there are two holomorphic disks having $c_{11}$ as a positive puncture and no negative puncture, which cancel each other. Due to its length, the only chord we can still use is $m_{12}$ and after this, no other chord can be added.
Let us show that this possible term for $\partial c_{11}$ is realized by exactly one Morse flow tree. We start at the chord $c_{11}$ and follow 
the negative gradient of the local height difference function, in the unique direction leading to the chord $c_{21}$. At this chord, we have 
a $2$-valent puncture of the Morse flow tree and we continue by following the negative gradient of the local height difference function corresponding to
the lowest two sheets, which is the Morse function used to perturb the Hopf link. Generically, this gradient trajectory will reach the minimum $m$ so
that we obtain a $1$-valent puncture of the Morse flow tree at $m_{12}$. The calculation of $\partial c_{22}$ is analogous.

For $\partial c_{12}$, there are no possible terms because no other chord can lead from $\Lambda_1$ to $\Lambda_2$. For $\partial M_{12}$,
the only chord which is short enough to appear is $m_{12}$ but its grading $k-1$ is strictly smaller when $n>1$ than the necessary grading $n+k-2$. 
When $n=1$, there are two gradient trajectories from the maximum to the minimum of a Morse function on the circle, which cancel each other. Finally, 
$\partial m_{12}=0$ because it is the shortest chord and it joins different components.
\end{proof}

\begin{Cor} \label{cor:aughopf}
If $k=1$, the Chekanov-Eliashberg DGA of $\Lambda^{(2)} \subset J^1(\R^n)$ has two augmentations $\varepsilon_L$ and $\varepsilon_R$,
such that $\varepsilon_L(m_{12}) = 0$, $\varepsilon_R(m_{12}) = 1$ and vanishing on the other Reeb chords. When $n>1$, there are no other 
augmentations. 
\end{Cor}

\begin{proof}
When $n>1$, $m_{12}$ is the only generator of degree $0$, so that the maps $\varepsilon_L$ and $\varepsilon_R$ are the only two degree preserving
algebra morphisms $\mathcal{A} \to \Z_2$. In order to show that these are augmentations, we need to check that $1, m_{12} \notin 
\textrm{im } \partial$. This follows from the fact that there is no term $1$ and that $m_{12}$ always appears a a factor of another generator 
in the expression of $\partial$ in Proposition~\ref{prop:hopf}.
\end{proof}

The above augmentations $\varepsilon_L$ and $\varepsilon_R$ can be used in order to obtain a bilinearized differential
associated to the differential from Proposition~\ref{prop:hopf}. We obtain 
$\partial^{\varepsilon_L, \varepsilon_R} c_{12}  = M_{12} + c_{22}$ and 
$\partial^{\varepsilon_L, \varepsilon_R} c_{11} = c_{21}$,
while the differential of the other $4$ generators vanishes. The corresponding homology is therefore generated by $[M_{12}] = -[c_{22}]$ in degree $n$ and by $[m_{12}]$ in degree $0$. Hence, the Poincar\'e polynomial $P_{\Lambda^{(2)},\varepsilon_L, \varepsilon_R}(t)$ is given by $1+t^n$.

After this preliminary calculation, let us consider a combination of several such links in view of obtaining more general Poincar\'e 
polynomials than those in Example~\ref{ex:trefoil}. To this end, consider the $2N$-copy of the standard Legendrian unknot 
$\Lambda^{(2N)}  \subset J^1(\R^n)$ for $N \ge 1$. This link contains the components $\Lambda_1, \ldots, \Lambda_{2N}$ 
numbered from bottom to top. If $\ell$ denotes the length of the unique Reeb chord of $\Lambda_i$ and $\varepsilon$ denotes the 
positive shift between any two consecutive components, we require that $2N \varepsilon$ is much smaller than $\ell$. 
We perturb the component $\Lambda_i$ for $i=2, \ldots, 2N$ by a Morse function $f_i$ with two critical points and of amplitude $\delta$
much smaller than $\varepsilon$, such that all differences $f_i - f_j$ with $i\neq j$ are Morse functions with two critical points.
In order to define the gradings of mixed Reeb chords in this link, we choose the Maslov potential of the component $\Lambda_i$ 
to be given by the Maslov potential of the lowest component $\Lambda_1$ plus $i-1$.

A direct application of Proposition~\ref{prop:hopf} to each pair of components $\Lambda_i$ and $\Lambda_j$ shows that the chords
of $\Lambda^{(2N)}$ are given by
\begin{center}
\begin{tabular}{c|cc}
& grading & length \\
\hline
$c_{i,i}$ & $n$ & $\ell$ \\
$c_{i,j}$ & $n+j-i$ & $\ell+\varepsilon(j-i)$ \\
$c_{j,i}$ & $n-j+i$ & $\ell-\varepsilon(j-i)$ \\
$m_{i,j}$ & $j-i-1$ & $\varepsilon(j-i)-\delta$ \\
$M_{i,j}$  & $n+j-i-1$ & $\varepsilon(j-i)+\delta$
\end{tabular}
\end{center}
where the indices $i$ and $j$ take all possible values between $1$ and $2N$, such that $i<j$. 

\begin{Prop}
The algebra morphisms $\varepsilon_L$ and $\varepsilon_R$ defined by $\varepsilon_L(m_{i,i+1})=1$ when $i$ is even,
$\varepsilon_R(m_{i,i+1})=1$ when $i$ is odd and vanishing on all other chords are augmentations of the Chekanov-Eliashberg
DGA of $\Lambda^{(2N)}$.
\end{Prop}

\begin{proof}
Let us to show that $m_{i,i+1} \notin \textrm{im } \partial$ for all $i=1,\ldots, 2N-1$. If $m_{i,i+1}$ was a term in $\partial a$ for some $a$
in the Chekanov-Eliashberg of $\Lambda^{(2N)}$, then $a$ would have to be a linear combination of chords from $\Lambda_i$ to $\Lambda_{i+1}$.
Indeed, $\partial c$ does not contain the term $1$ for any chord $c$ of $\Lambda^{(2N)}$, say from $\Lambda_i$ to $\Lambda_j$, because it would 
give rise to a term $1$ in Proposition~\ref{prop:hopf} for the Legendrian Hopf link composed of $\Lambda_i$ and $\Lambda_j$. Therefore, $\partial$
does not decrease the number of factors in terms it acts on. Since $a$ must be a single chord from $\Lambda_i$ to $\Lambda_{i+1}$, if there were
a  term $m_{i,i+1}$ in $\partial a$, then there would already be such a term in Proposition~\ref{prop:hopf} for the Legendrian Hopf link composed of 
$\Lambda_i$ and $\Lambda_{i+1}$. Hence $m_{i,i+1} \notin \textrm{im } \partial$ as announced. 

This implies that $\varepsilon_L$ and $\varepsilon_R$ are augmentations, because any element of $\textrm{im } \partial$ is composed of monomials
having at least one factor which is not of the form $m_{i,i+1}$, and in particular not augmented, so that $\varepsilon_L$ and $\varepsilon_R$ vanish on
$\textrm{im } \partial$.
\end{proof}

\begin{Prop}  \label{prop:bilinN}
The bilinearized differential $\partial^{\varepsilon_L,\varepsilon_R}$ of $\Lambda^{(2N)}$ is given by
\begin{eqnarray*}
\partial^{\varepsilon_L,\varepsilon_R} c_{i,i} &=& \overline{i} \ c_{i,i-1} +  \overline{i} \ c_{i+1,i}, \\
\partial^{\varepsilon_L,\varepsilon_R} c_{i,j} &=& M_{i,j} + \overline{j} \ c_{i,j-1} + \overline{i} \ c_{i+1,j}, \\
\partial^{\varepsilon_L,\varepsilon_R} c_{j,i} &=& \overline{i} \ c_{j,i-1} + \overline{j} \ c_{j+1,i}, \\
\partial^{\varepsilon_L,\varepsilon_R} m_{i,j} &=& \overline{j} \ m_{i,j-1} + \overline{i} \ m_{i+1,j}, \\
\partial^{\varepsilon_L,\varepsilon_R} M_{i,j} &=& \overline{j} \ M_{i,j-1} + \overline{i} \ M_{i+1,j}, 
\end{eqnarray*}
with $i < j$ and where $\overline{i}$ and $\overline{j}$ are the modulo $2$ reductions of $i$ and $j$.
In the above formulas, any generator with one of its indices equal to $0$ or $2N+1$ or of the form $m_{i,i}$ or $M_{i,i}$ 
should be replaced by zero.
\end{Prop}

\begin{proof}
The link $\Lambda^{(2N)}$ and its Reeb chords determine a quiver represented in Figure~\ref{fig:quiverN}, and as in the proof 
of Proposition~\ref{prop:hopf}, the terms in the differential of a chord from $\Lambda_i$ to $\Lambda_j$ must form a path from 
vertex $i$ to vertex $j$.

\begin{figure}
\labellist
\small\hair 2pt
\pinlabel {$\Lambda_1$} [bl] at 33 53
\pinlabel {$\Lambda_2$} [bl] at 141 48
\pinlabel {$\Lambda_3$} [bl] at 260 50
\pinlabel {$\Lambda_{2N-2}$} [bl] at 338 50
\pinlabel {$\Lambda_{2N-1}$} [bl] at 467 48
\pinlabel {$\Lambda_{2N}$} [bl] at 590 53
\endlabellist
  \centerline{\includegraphics[width=12cm]{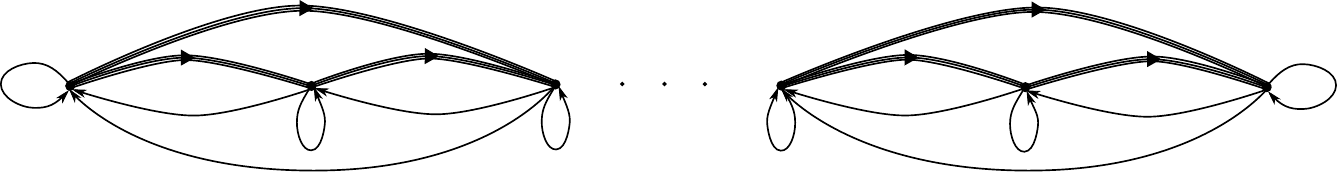}}
  \caption{Quiver corresponding to the $2N$-copy of the standard Legendrian unknot.}
  \label{fig:quiverN}
\end{figure}

Let us compute $\partial^{\varepsilon_L,\varepsilon_R} c_{i,i}$. The only possible terms in $\partial c_{i,i}$ that could lead to a nonzero contribution to
$\partial^{\varepsilon_L,\varepsilon_R} c_{i,i}$ are $c_{i+1,i} m_{i,i+1}$ and  $m_{i-1,i} c_{i,i-1}$. Indeed, there are no other chords of $\Lambda_i$ so
a change of component is needed. Since only chords of the form $m_{i,i+1}$ are augmented by $\varepsilon_L$ and $\varepsilon_R$, there must be 
exactly one chord from $\Lambda_j$ to $\Lambda_k$ with $j > k$. Moreover, since neither $\varepsilon_L$ nor $\varepsilon_R$ augment consecutive
chords in the quiver determined by $\Lambda^{(2N)}$, we must have $|j-k|=1$ and $j=i$ or $k=i$. Considering the Legendrian Hopf link composed of
$\Lambda_i$ and $\Lambda_{i+1}$, Proposition~\ref{prop:hopf} gives the term $c_{i+1,i} m_{i,i+1}$, while considering the Legendrian Hopf link composed 
of $\Lambda_{i-1}$ and $\Lambda_i$, it gives the term $m_{i-1,i} c_{i,i-1}$. With the first term, since $m_{i,i+1}$ has to be augmented by $\varepsilon_R$,
we obtain the contribution $c_{i+1,i}$ when $i$ is odd.  With the second term, since $m_{i-1,i}$ has to be augmented by $\varepsilon_L$,
we obtain the contribution $c_{i,i-1}$ when $i-1$ is even. In other words, we obtain $\partial^{\varepsilon_L,\varepsilon_R} c_{i,i}= 
\overline{i} \ c_{i,i-1} +  \overline{i} \ c_{i+1,i}$ as announced.

Let us compute $\partial^{\varepsilon_L,\varepsilon_R} c_{i,j}$ with $i<j$. All terms in $\partial c_{i,j}$ involving a single chord from $\Lambda_i$ to 
$\Lambda_j$ correspond to terms with a single factor in the expression for $\partial c_{12}$ in Proposition~\ref{prop:hopf}. We therefore obtain the
term $M_{i,j}$. The other terms must involve augmented chords; since $\varepsilon_L$ and $\varepsilon_R$ do not have consecutive augmented 
chords, these other terms could come from $m_{j-1,j} c_{i,j-1}$, $c_{i+1,j} m_{i,i+1}$, $m_{j-1,j} c_{i+1,j-1} m_{i,i+1}$ or analogous terms with $c_{k,l}$
replaced with $m_{k,l}$ or $M_{k,l}$. The latter two possibilities lead to elements with a too small grading, so that the unaugmented chord is of the type
$c_{k,l}$. The possibilities $m_{j-1,j} c_{i,j-1}$ and $c_{i+1,j} m_{i,i+1}$ are each realized by a single holomorphic disk, corresponding to the contribution
$m_{12} c_{11} + c_{22} m_{12}$ in the expression for $\partial c_{12}$ in Proposition~\ref{prop:hopf}. The remaining possibility $m_{j-1,j} c_{i+1,j-1} m_{i,i+1}$ has a too small grading. Summing up, the possibility $m_{j-1,j} c_{i,j-1}$ leads to the term $c_{i,j-1}$ when $j$ is odd and the possibility 
$c_{i+1,j} m_{i,i+1}$ leads to the term $c_{i+1,j}$ when $i$ is odd, so that we obtain $\partial^{\varepsilon_L,\varepsilon_R} c_{i,j} = M_{i,j} + 
\overline{j} \ c_{i,j-1} + \overline{i} \ c_{i+1,j}$ as announced.

The computation of $\partial^{\varepsilon_L,\varepsilon_R} c_{j,i}$ with $i<j$ is similar. Since there are no other chords from $\Lambda_i$ to 
$\Lambda_j$, the only contributions involve augmented chords and come from $m_{i-1,i} c_{j,i-1}$, $c_{j+1,i} m_{j,j+1}$ or $m_{i-1,i} c_{j+1,i-1} m_{j,j+1}$.
The last possibility has a too small grading, while the first two possibilities are each realized by a single holomorphic disk, corresponding to the 
contributions $c_{21} m_{12}$  and $m_{12} c_{21}$ in the expressions for $\partial c_{11}$ and $\partial c_{22}$ in Proposition~\ref{prop:hopf}.
The possibility $m_{i-1,i} c_{j,i-1}$ leads to the term $c_{j,i-1}$ when $i$ is odd and the possibility 
$c_{j+1,i} m_{j,j+1}$ leads to the term $c_{j+1,i}$ when $j$ is odd, so that we obtain $\partial^{\varepsilon_L,\varepsilon_R} c_{j,i} = 
\overline{i} \ c_{j,i-1} + \overline{j} \ c_{j+1,i}$ as announced.

The computation of $\partial^{\varepsilon_L,\varepsilon_R} m_{i,j}$ and $\partial^{\varepsilon_L,\varepsilon_R} M_{i,j}$ with $i<j-1$ involves only chords
of the type $m_{k,l}$ and $M_{k,l}$ since all other chords are much longer. Let us start with $\partial^{\varepsilon_L,\varepsilon_R} m_{i,j}$. Arguing as above, since $m_{i,j}$ is the shortest chord from $\Lambda_i$ to $\Lambda_j$, the only contributions involve augmented chords and come from 
$m_{i-1,i} m_{j,i-1}$, $m_{j+1,i} m_{j,j+1}$ or $m_{i-1,i} m_{j+1,i-1} m_{j,j+1}$. The last possibility has a too small grading, and the first two possibilities 
are each realized by a unique Morse flow tree~\cite{E}, which in turn corresponds to a unique holomorphic curve. Both Morse flow trees start with a
constant gradient trajectory at $m_{i,j}$, which is the minimum of the difference function $f_j-f_i$. The only possibility to leave $m_{i,j}$ is to have a 
$3$-valent vertex, corresponding to the splitting of the gradient trajectory into two gradient trajectories, for $f_j - f_k$ and for $f_k - f_i$, for some $k$
strictly between $i$ and $j$. These trajectories converge to the corresponding minima $m_{k,j}$ and to $m_{i,k}$, so we obtain the desired trees for
$k = i+1$ and $k = j-1$. Summing up, we obtain as above $\partial^{\varepsilon_L,\varepsilon_R} m_{j,i} = 
\overline{i} \ m_{j,i-1} + \overline{j} \ m_{j+1,i}$ as announced. The computation of $\partial^{\varepsilon_L,\varepsilon_R} M_{i,j}$  is completely 
analogous, except for the description of the Morse flow trees. Both Morse flow trees start with a gradient trajectory from $M_{i,j}$ to a priori any 
point of the sphere. In order to reach $M_{i+1,j}$ or $M_{i,j-1}$ it is necessary for the gradient trajectory to end exactly at the maximum of the
corresponding height difference function. There, we have a $2$-valent puncture of the Morse flow tree and we continue with a gradient trajectory
converging to the minimum $m_{i,i+1}$ or $m_{j-1,j}$. Again, we obtain $\partial^{\varepsilon_L,\varepsilon_R} M_{j,i} = 
\overline{i} \ M_{j,i-1} + \overline{j} \ M_{j+1,i}$ as announced.
\end{proof}

\begin{Prop}  \label{prop:polyN}
The Poincar\'e polynomial of $\Lambda^{(2N)}$ with respect to the augmentations $\varepsilon_L$ and  $\varepsilon_R$
is given by $P_{\Lambda^{(2N)}, \varepsilon_L, \varepsilon_R}(t) = N(1 + t^n)$.
\end{Prop}

\begin{proof}
We need to compute the homology of the complex described in Proposition~\ref{prop:bilinN}. 

Let us first consider the subcomplex
spanned by the chords $m_{i,j}$ with $i<j$. For any $k,l = 1, \ldots, N$ with $k < l-1$, the generators $m_{2k-1,2l-1}$, $m_{2k,2l-1}$,
$m_{2k-1,2l-2}$ and $m_{2k,2l-2}$ form an acyclic subcomplex. When $k=l-1$, we just have a subcomplex with the $3$ generators
$m_{2l-3,2l-1}, m_{2l-2,2l-1}$ and $m_{2l-3,2l-2}$, which has homology spanned by $[m_{2l-2,2l-1}]=[m_{2l-3,2l-2}]$ in degree $0$.
We therefore obtain $N-1$ such generators. For any $k = 1, \ldots, N-1$, the generators $m_{2k-1,2N}$ and $m_{2k,2N}$
form an acyclic subcomplex. Finally, the generator $m_{2N-1,2N}$ survives in homology and has degree $0$. The total contribution of
the chords $m_{i,j}$ to the polynomial $P_{\Lambda^{(2N)}, \varepsilon_L, \varepsilon_R}$ is therefore the term $N$.

Consider now the subcomplex spanned by the chords $M_{i,j}$ with $i<j$ and $c_{i,j}$ for all $i,j = 1, \ldots, 2N$. For any $k,l = 1, \ldots, N$
with $k<l-1$, the generators $c_{2k-1,2l-1}$, $c_{2k,2l-1}$, $c_{2k-1,2l-2}$, $c_{2k,2l-2}$, $M_{2k-1,2l-1}$, $M_{2k,2l-1}$, $M_{2k-1,2l-2}$ and 
$M_{2k,2l-2}$ form an acyclic subcomplex. When $k=l-1$, we just have a subcomplex with the $7$ generators
$c_{2k-1,2l-1}$, $c_{2k,2l-1}$, $c_{2k-1,2l-2}$, $c_{2k,2l-2}$, $M_{2k-1,2l-1}$, $M_{2k,2l-1}$ and $M_{2k-1,2l-2}$, which has homology spanned 
by $c_{2l-2,2l-2}$ in degree $n$. We therefore obtain $N-1$ such generators. For any $k = 1, \ldots, N-1$, the generators $c_{2k-1,2N}$,$c_{2k,2N}$, 
$M_{2k-1,2N}$ and $M_{2k,2N}$ form an acyclic subcomplex. But the subcomplex spanned by the $3$ generators $c_{2N-1,2N}$, $c_{2N,2N}$,
$M_{2N-1,2N}$ has homology generated by $[c_{2N,2N}]=[M_{2N-1,2N}]$ in degree $n$. For any $k,l = 1, \ldots, N$ with $k \le l$ and $k>1$, 
the generators $c_{2l-1,2k-1}$, $c_{2l,2k-1}$, $c_{2l-1,2k-2}$ and $c_{2l,2k-2}$ form an acyclic subcomplex. When $k=1$, we just have an acyclic
subcomplex with the $2$ generators $c_{2l-1,1}$ and $c_{2l,1}$. The total contribution of
the chords $M_{i,j}$ with $i<j$ and $c_{i,j}$ to the polynomial $P_{\Lambda^{(2N)}, \varepsilon_L, \varepsilon_R}$ is therefore the term $N t^n$.

The sum of the above two contributions therefore gives $P_{\Lambda^{(2N)}, \varepsilon_L, \varepsilon_R}(t) = N(1 + t^n)$
as announced.
\end{proof}

The next step is to perform some type of connected sum on the Legendrian link $\Lambda^{(2N)}$ in order to obtain a Legendrian sphere 
$\widetilde{\Lambda}^{(2N)} \subset J^1(\R^n)$. More precisely, for each $i = 1, \ldots, N-1$, we consider the Legendrian link 
formed by $\Lambda_{2i-1}, \Lambda_{2i}, \Lambda_{2i+1}$ and $\Lambda_{2i+2}$ as the $2$-copy of the Legendrian link formed 
by $\Lambda_{2i-1}$ and $\Lambda_{2i+1}$, and we perform the $2$-copy of the connected sum of $\Lambda_{2i-1}$ and $\Lambda_{2i+1}$ as follows.

We now describe the connected sum of $\Lambda_{2i-1}$ and $\Lambda_{2i+1}$ in more detail. We deform $\Lambda_{2i-1}$ by a 
Legendrian isotopy corresponding to the spinning of two iterated first Reidemeister moves on one half of the standard Legendrian unknot in $J^1(\R)$. 
Since this front in $J^0(\R)$ has a vertical symmetry axis, we can spin it around this axis to produce a Legendrian surface in 
$J^1(\R^2)$ as in~\cite[Section~3.2]{BST}. The resulting front has vertical symmetry planes and hence is spinnable around such a plane; iterating 
the spinning construction, we obtain the desired $2$-components Legendrian link in $J^1(\R^n)$ with cusp edges from (the deformation of) 
$\Lambda_{2i-1}$ and $\Lambda_{2i+1}$ facing each other and having the same Maslov potentials. This is illustrated by 
Figure~\ref{fig:connsum_highdim}. 

\begin{figure}
  \centerline{\includegraphics[width=8cm, height=4cm]{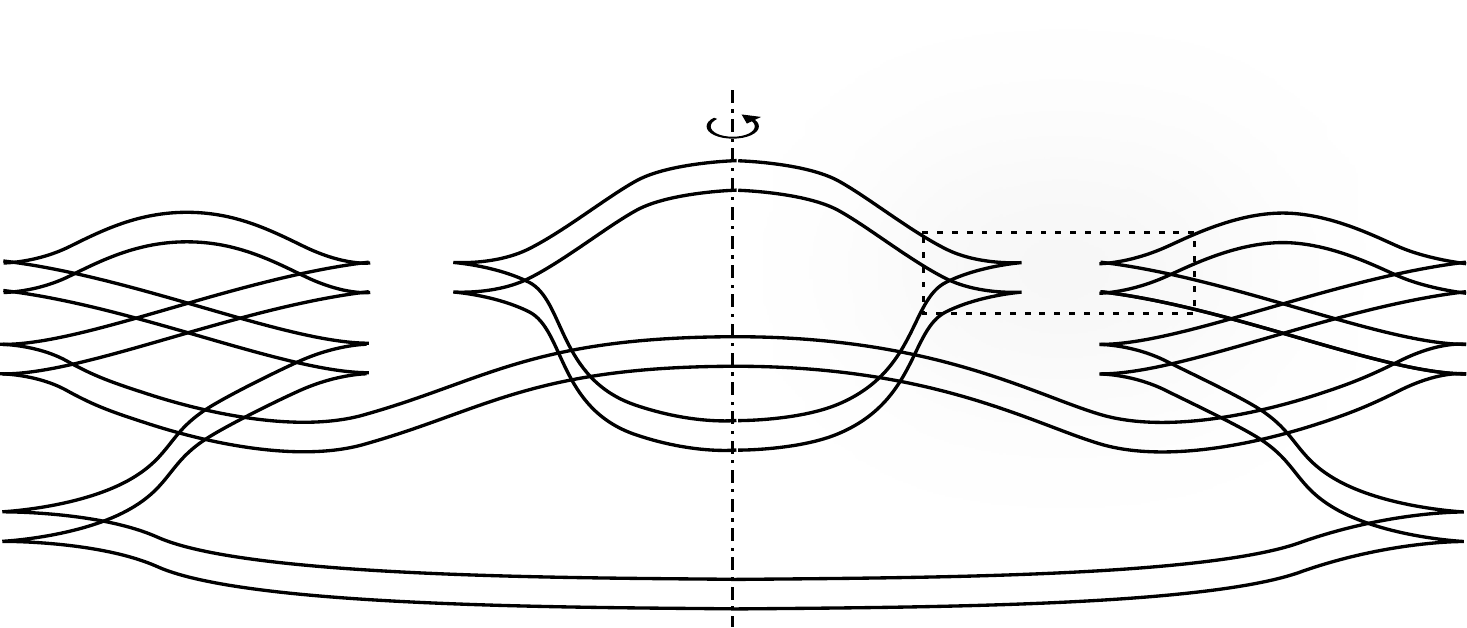}}
  \caption{Isotopy of $\Lambda_{2i-1}, \Lambda_{2i}, \Lambda_{2i+1}$ and $\Lambda_{2i+2}$.}
  \label{fig:connsum_highdim}
\end{figure}

\begin{figure}
  \centerline{\includegraphics[width=10cm, height=3cm]{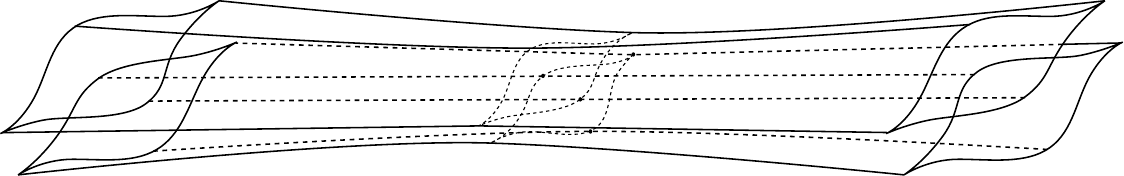}}
  \caption{Double tube.}
  \label{fig:doubletube}
\end{figure}

On this Figure, we consider the rectangular area limited by a dashed line: its image in $J^0(\R^+) \subset J^0(\R^n)$, 
i.e. with all spinning angles set to zero, is a rectangular area intersecting $\Lambda_{2i-1}, \Lambda_{2i}, \Lambda_{2i+1}$ 
and $\Lambda_{2i+2}$ in the 2-copy of two cusps facing each other. We then replace a neighborhood of this rectangular area
with the 2-copy of a connecting tube, as shown in Figure~\ref{fig:doubletube}. This operation is equivalent to the 2-copy
of the connected sum operation described in~\cite[Section~4]{BST}.

Finally, after performing $N-1$ times these 2-copies of connected sums, we are left with a Legendrian link composed of two connected components: 
$\Lambda_{\rm odd}$ resulting from the connected sum of $\Lambda_{2i-1}$ for $i=1, \ldots, N$ and $\Lambda_{\rm even}$ 
resulting from the connected sum 
of $\Lambda_{2i}$ for $i=1, \ldots, N$. We then perform an (ordinary) connected sum between these components in order to obtain
the Legendrian sphere $\widetilde{\Lambda}^{(2N)}$.

\begin{Prop} \label{prop:augtube}
The augmentations $\varepsilon_L$ and $\varepsilon_R$ of $\Lambda^{(2N)}$ induce augmentations $\widetilde{\varepsilon}_L$ 
and $\widetilde{\varepsilon}_R$ of $\widetilde{\Lambda}^{(2N)}$. 
\end{Prop}

\begin{proof}
Note that it suffices to show that an augmentation induces another augmentation after a single 2-copy of a connected sum.  
To this end, we describe this operation differently, in order to gain a better control on the Reeb chords during this process.
Before performing the 2-copy connected sum connecting 
$\Lambda_{2i-1}$ and $\Lambda_{2i}$ to $\Lambda_{2i+1}$ and $\Lambda_{2i+2}$ respectively, we deform these components by a Legendrian
isotopy in order to create a pair of canceling critical points $m'_{2i-1,2i}$ of index $0$ and $s_{2i-1,2i}$ of index $1$ for the Morse function $f_{2i} - f_{2i-1}$ 
and a similar pair $m'_{2i+1,2i+2}$, $s_{2i+1,2i+2}$ for $f_{2i+2} - f_{2i+1}$ near the attaching locus of the connecting double tube. More precisely, the chords
$m'_{2i-1,2i}$ and $m'_{2i+1,2i+2}$ are contained in the small balls that are removed during the connected sums, while the chords $s_{2i-1,2i}$ and 
$s_{2i+1,2i+2}$ are just outside these balls. The connecting double tube is the thickening of an $n-1$-dimensional standard Legendrian Hopf link, and we
shape each tube so that its thickness in the $z$-direction is minimal in the middle. We extend the Morse functions $f_{2i} - f_{2i-1}$ and $f_{2i+2} - f_{2i+1}$
by a Morse function on the connecting tube decreasing towards its middle and having exactly two critical points (of index $0$ and $n-1$) in its middle slice.
All Reeb chords for the connecting double tube are contained in this middle slice and correspond to the generators described in Proposition~\ref{prop:hopf}
with $k=1$ and $n$ replaced with $n-1$:

\begin{center}
\begin{tabular}{c|cc}
& grading & length \\
\hline
$c^h_{2i-1,2i-1}$ & $n-1$ & $\ell' < \ell$ \\
$c^h_{2i,2i}$ & $n-1$ & $\ell'$ \\
$c^h_{2i-1,2i}$ & $n$ & $\ell'+\varepsilon$ \\
$c^h_{2i,2i-1}$ & $n-2$ & $\ell'-\varepsilon$ \\
$m^h_{2i-1,2i}$ & $0$ & $\varepsilon-\delta$ \\
$M^h_{2i-1,2i}$  & $n-1$ & $\varepsilon+\delta$
\end{tabular}
\end{center}

The last two generators correspond to the critical points of the Morse function on the connecting tube mentioned above. The unital subalgebra $\mathcal{A}^h$ generated by
these 6 generators is a subcomplex of the Chekanov-Eliashberg DGA, because Morse-flow trees are pushed towards the middle of the double connecting tube
due to its shape. By Corollary~\ref{cor:aughopf}, this subcomplex has two augmentations such that only $m^h_{2i-1,2i}$ is possibly augmented. On the other
hand, we have $\partial s_{2i-1,2i} = m_{2i-1,2i} + m'_{2i-1,2i}$ with no other terms because the length of $s_{2i-1,2i}$ is very short. Hence, for any augmentation
$\varepsilon$, we must have $\varepsilon(m'_{2i-1,2i}) = \varepsilon(m_{2i-1,2i})$ and this forces the choice of the augmentation for $\mathcal{A}^h$. More precisely, the map $\widetilde{\varepsilon}$ induced by $\varepsilon$ must satisfy 
$\widetilde{\varepsilon}(m^h_{2i-1,2i}) = \varepsilon(m_{2i-1,2i})$. Similarly, arguing with $s_{2i+1,2i+2}$, we also have
$\widetilde{\varepsilon}(m^h_{2i-1,2i}) = \varepsilon(m_{2i+1,2i+2})$. Note that these relations are compatible since each of 
$\varepsilon_L$ and $\varepsilon_R$ have the same value on $m_{2i-1,2i}$ and $m_{2i+1,2i+2}$.

Let us check that the resulting maps $\widetilde{\varepsilon}_L, \widetilde{\varepsilon}_R : \mathcal{A}(\widetilde{\Lambda}^{(2N)}) \to \mathbb{Z}_2$ satisfy 
$\widetilde{\varepsilon}_L \circ \partial = 0 = \widetilde{\varepsilon}_R \circ \partial$. We already saw that these relations are satisfied on $\mathcal{A}^h$ as well as on $s_{2i-1,2i}$ and $s_{2i+1,2i+2}$. 
On any other chord $c$, the relation was satisfied before the 2-copy of connected sum.
We claim that the augmented terms in $\partial c$ are modified by the 2-copy of connected sum in the following way:
all occurrences of $m'_{2i-1,2i}$ and $m'_{2i+1,2i+2}$ are replaced with $m^h_{2i-1,2i}$. In particular, the maps
$\widetilde{\varepsilon}_L, \widetilde{\varepsilon}_R$
keep the same value on these terms and the augmentation relation continues to hold after the 2-copy of connected sum.

To verify the claim, note that the region in which the 2-copy of connected sum is taking place is a trap for Morse-flow trees:
any portion of such a tree entering this region cannot leave it, because all relevant gradient vector fields are pointing inwards.
We only have to consider augmented terms, since these are the only ones that could harm the augmentation relation.
We first consider an augmented term that does not contain $m'_{2i-1,2i}$ nor $m'_{2i+1,2i+2}$. If the corresponding Morse
flow tree enters the region in which the 2-copy of connected sum is taking place, it must end at a cusp edge. Moreover, it
cannot contain any trivalent vertex, otherwise it would not be rigid. Hence, it is a single gradient trajectory ending at a cusp 
edge. After the 2-copy of connected sum, it becomes another gradient trajectory, also ending at a cusp edge.
Hence the corresponding term is not affected by the 2-copy of connected sum.
Consider now an augmented term containing $m'_{2i-1,2i}$ or $m'_{2i+1,2i+2}$. A rigid Morse-flow tree cannot have 
a 2-valent negative puncture at such a chord, since it is a minimum of the Morse function $f_{2i} - f_{2i-1}$ or 
$f_{2i+2} - f_{2i+1}$~\cite[Lemma~3.7]{E}, so that these chords are 1-valent negative punctures. The only other way
a fragment of Morse-flow tree contained in the region in which the 2-copy of connected sum is taking place can end is
at a cusp edge. As above, it cannot contain any trivalent vertex, otherwise it would not be rigid. Hence, it is a single
gradient trajectory ending at a minimum $m'_{2i-1,2i}$ or $m'_{2i+1,2i+2}$. After the 2-copy of connected sum, it 
becomes another gradient trajectory, also ending at a minimum $m^h_{2i-1,2i}$. 
Conversely, consider an augmented term containing $m^h_{2i-1,2i}$ after the 2-copy of connected sum. In particular,
the corresponding Morse-flow tree can only end at the chord $m^h_{2i-1,2i}$ (at a 1-valent negative puncture, as above) 
or at a cusp edge. For the same reason as above, such a rigid tree cannot contain a trivalent vertex in the 2-copy of
the connecting tube. Hence, it is just a single gradient trajectory ending at $m^h_{2i-1,2i}$. If we remove the 2-copy 
of the connecting tube and replace it with the regions containing the minima $m'_{2i-1,2i}$ and $m'_{2i+1,2i+2}$, this
gradient trajectory is replaced with a single gradient trajectory ending at one of these minima. In other words, such an
augmented term involving $m^h_{2i-1,2i}$ always comes from the substitution of $m'_{2i-1,2i}$ and $m'_{2i+1,2i+2}$  
with $m^h_{2i-1,2i}$, proving the claim.
\end{proof}

We are now in position to show that these 2-copies of connected sums destroy almost all terms in the Poincar\'e polynomial for bilinearized LCH.

\begin{Prop}  \label{prop:polyNtilde}
The Poincar\'e polynomial $P_{\widetilde{\Lambda}^{(2N)}, \widetilde{\varepsilon}_L, \widetilde{\varepsilon}_R}$ is equal to $1$.
\end{Prop}

\begin{proof}
Let us show by induction that, after applying $i$ successive $2$-copies of connected sums on $\Lambda^{(2N)}$, 
its Poincar\'e polynomial is given by $(N-k)(1+t^n)$. Proposition~\ref{prop:polyN} corresponds to the case $i=0$. 
Let us denote for shorthand notation $C_*$ the bLCH chain complex after $i-1$ successive $2$-copies of connected sums,
and $\tilde{C}_*$ the bLCH chain complex after $i$ successive $2$-copies of connected sums. Using the description of 
the $i$th 2-copy of connected sum in the proof of Proposition~\ref{prop:augtube}, we see that this operation has two
effects on the complex $C_*$. First, the generators $m'_{2i-1,2i}$ and $m'_{2i+1,2i+2}$ are removed. Second, we add 
generators of the bLCH complex $C^h_*$ of the $(n-1)$-dimensional standard Legendrian Hopf link with distinct 
augmentations. Recall that $C^h_*$ forms a subcomplex of $\tilde{C}_*$, as explained in the proof of
Proposition~\ref{prop:augtube}. 

Since the $2$-copy of connected sum is performed away from rigid holomorphic disks connecting generators
of $\tilde{C}_*/C^h_*$, the differential on this quotient complex is directly induced from that of $C_*$. In particular, we have
$\partial s_{2i-1,2i} = m_{2i-1,2i}$ and $\partial s_{2i+1,2i+2} = m_{2i+1,2i+2}$ in $\tilde{C}_*/C^h_*$. 
Hence, its homology
coincides with the homology of $C_*$, except in degree $0$, where it has $2$ fewer generators. Hence, its Poincar\'e 
polynomial is $(N-i-1) + (N-i+1)t^n$.
On the other hand, the homology of $C^h_*$ is given by Proposition~\ref{prop:polyN} with $N=1$ and $n$ replaced
with $n-1$. Hence its Poincar\'e polynomial is $1 + t^{n-1}$. 

In order to deduce the homology of $\tilde{C}_*$, we consider the long exact sequence
$$
\ldots \!\to\! H_{k+1}(\tilde{C}_*/C^h_*) \!\to\! H_k(C^h_*) \!\to\! H_k(\tilde{C}_*) \!\to\! H_k(\tilde{C}_*/C^h_*) \!\to\! 
H_{k-1}(C^h_*) \!\to\! \ldots
$$
When $k=0$, we see that the generator $[m^h_{2i-1,2i}]$ of $H_0(C^h_*)$ injects into
$H_0(\tilde{C}_*)$, as it can only be hit by $s_{2i-1,2i}$ and $s_{2i+1,2i+2}$, but these do not survive
in the homology of the quotient complex. Hence the rank of $H_0(\tilde{C}_*)$ is $N-i$.

When $k=n$, we see that the generator $[c_{2i+2,2i+2}]$ in $H_n(\tilde{C}_*/C^h_*)$, which was not affected by 
the $i-1$ first $2$-copies of connected sums, hits the generator $[c^h_{2i,2i}]$ of $H_{n-1}(C^h_*)$, because there 
exists a single Morse flow tree connecting them. Indeed, on Figure~\ref{fig:connsum_highdim}, the chord $c_{2i+2,2i+2}$
is in the middle of the uppermost connected component, and the Morse flow tree starts from there to the right in the plane
of the Figure (corresponding to all spinning angles set to zero), then enters the dotted rectangle, hence the upper tube
in Figure~\ref{fig:doubletube}, until it reaches the chord $c^h_{2i,2i}$ sitting in the middle of that tube. Hence,
the rank of $H_n(\tilde{C}_*)$ is $N-i$.
The Poincar\'e polynomial for the homology of $\tilde{C}_*$ is therefore $(N-i)(1+t^n)$ as announced.

After these $N-1$ operations,
we are therefore left with the Poincar\'e polynomial $1+t^n$. The last step in the construction of $\widetilde{\Lambda}^{(2N)}$ is an ordinary connected sum 
between the remaining two connected components $\Lambda_{\rm even}$ (the connected sum of $\Lambda_{2i}$ for $i = 1, 
\ldots, N$) and $\Lambda_{\rm odd}$ (the connected sum of $\Lambda_{2i-1}$ for $i = 1, \ldots, N$). Let us denote the 
corresponding $2$-component Legendrian link by $\Lambda'$.

As in the proof of Proposition~\ref{prop:tau}, the map $\tilde{\tau}_0$ from the duality exact sequence~\eqref{eq:bilindualityseq} with
$\varepsilon_1 = \widetilde{\varepsilon}_R$ and $\varepsilon_2 = \widetilde{\varepsilon}_L$ is given at chain level by 
$\widetilde{\varepsilon}_R - \widetilde{\varepsilon}_L$, except that we have to refine according to the connected
component $\Lambda_{\rm even}$ or $\Lambda_{\rm odd}$ which is hit. Note that all chords augmented by $\widetilde{\varepsilon}_L$
are starting on $\Lambda_{\rm odd}$ and all chords augmented by $\widetilde{\varepsilon}_R$ are ending on $\Lambda_{\rm odd}$.
This means that $\tilde{\tau}_0$ necessarily takes its values in $H_0(\Lambda_{\rm odd})$. By Proposition~\ref{prop:relpoly}, since
$P_{\Lambda', \widetilde{\varepsilon}_L, \widetilde{\varepsilon}_R}(t) = 1+t^n$ and $H_*(\Lambda')$ has rank $4$, 
we must have $p=0$ and hence
$P_{\Lambda', \widetilde{\varepsilon}_R, \widetilde{\varepsilon}_L}(t) = 1+t^n$ as well. Therefore, the image of the map
$\tilde{\tau}_0 : LCH_0^{\widetilde{\varepsilon}_R, \widetilde{\varepsilon}_L}(\Lambda') \to H_0(\Lambda')$ is equal to 
$H_0(\Lambda_{\rm odd})$.

We deduce that $\ker \tilde{\sigma}_n = H_0(\Lambda_{\rm odd})$ in the duality exact sequence~\eqref{eq:bilindualityseq}
with $\varepsilon_1 = \widetilde{\varepsilon}_R$ and $\varepsilon_2 = \widetilde{\varepsilon}_L$.
Consider now the map $\tau_n$ in the duality exact sequence~\eqref{eq:bilindualityseq}
with $\varepsilon_1 = \widetilde{\varepsilon}_L$ and $\varepsilon_2 = \widetilde{\varepsilon}_R$.
Since $\tilde{\sigma}_n$ and $\tau_n$ are adjoint in the sense of~\cite[Proposition 3.9]{EESdual}, $\textrm{im } \tau_n$ is the
annihilator of $H_0(\Lambda_{\rm odd})$ for the intersection pairing, which is $H_n(\Lambda_{\rm even})$. In particular, 
the map $\tau_{n,1} - \tau_{n,2} = \tau_{n,\rm odd} - \tau_{n,\rm even}$ from Proposition \ref{prop:connsum} does not vanish, 
so that this last connected sum modifies the Poincar\'e polynomial by $-t^n$. 
We are therefore left with $P_{\widetilde{\Lambda}^{(2N)}, \widetilde{\varepsilon}_L, \widetilde{\varepsilon}_R}(t) = 1$ 
as announced.
\end{proof}

%
%
\subsection{Geography of bLCH for Legendrian spheres}

The next step in our construction is to add to $\widetilde{\Lambda}^{(2N)}$ a standard Legendrian unknot $\Lambda_0$ which forms with the bottom
$k$ components $\Lambda_1, \ldots, \Lambda_k$ a Legendrian link isotopic to the $k+1$-copy of the standard Legendrian unknot,
but which is unlinked with the $2N-k$ top components $\Lambda_{k+1}, \ldots, \Lambda_{2N}$. We fix the Maslov potential of the component 
$\Lambda_0$ to be given by the Maslov potential of $\Lambda_1$ plus $m-1$, for some integer $m$. We can deform this link by a Legendrian 
isotopy in order to widen the components $\Lambda_1, \ldots, \Lambda_k \subset J^1(\R^n)$ so that their projection to $\R^n$ becomes much 
larger than the projection of the components $\Lambda_{k+1}, \ldots, \Lambda_{2N}$. We further narrow the component $\Lambda_0$ so that its
projection to $\R^n$ does not intersect the projection of the components $\Lambda_{k+1}, \ldots, \Lambda_{2N}$. We denote the resulting Legendrian
link by $\widetilde{\Lambda}^{(2N)}_{(k,m)}$. 

The addition of $\Lambda_0$ to $\widetilde{\Lambda}^{(2N)}$  is illustrated by Figure~\ref{fig:partitiononenote} in the case $k=4$,
where the picture zooms on the bottom strata of the $k$ components $\Lambda_1, \ldots, \Lambda_k$, which are represented
as portions of horizontal planes.

\begin{figure}
  \centerline{\includegraphics[width=7cm]{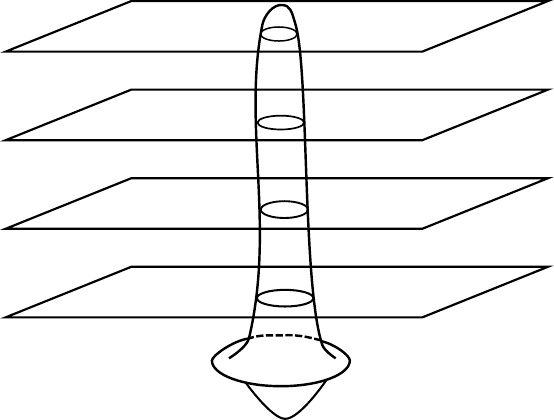}}
  \caption{Additional component $\Lambda_0$ with $k=4$.}
  \label{fig:partitiononenote}
\end{figure}

This Legendrian link $\widetilde{\Lambda}^{(2N)}_{(k,m)}$ has several additional Reeb chords compared to $\widetilde{\Lambda}^{(2N)}$. These are
easily identified within the $k+1$-copy of the standard Legendrian unknot formed by $\Lambda_0, \Lambda_1, \ldots, \Lambda_k$ and are given by
\begin{center}
\begin{tabular}{c|cc}
& grading & length \\
\hline
$c_{0,0}$ & $n$ & $\ell$ \\
$c_{0,j}$ & $n+j-m$ & $\ell+\varepsilon j$ \\
$c_{j,0}$ & $n-j+m$ & $\ell-\varepsilon j$ \\
$m_{0,j}$ & $j-m-1$ & $\varepsilon j-\delta$ \\
$M_{0,j}$  & $n+j-m-1$ & $\varepsilon j+\delta$
\end{tabular}
\end{center}
where the index $j$ takes all possible values between $1$ and $k$. 

We extend the augmentations $\widetilde{\varepsilon}_L$ and $\widetilde{\varepsilon}_R$ by zero on these additional chords in order to define 
augmentations, still denoted by $\widetilde{\varepsilon}_L$ and $\widetilde{\varepsilon}_R$, on the Chekanov-Eliashberg DGA of 
$\widetilde{\Lambda}^{(2N)}_{(k,m)}$. Since the mixed chords involving $\Lambda_0$ are not augmented, it follows that the vector space generated 
by the above chords is a direct summand of the bilinearized complex with respect to the differential 
$\partial^{\widetilde{\varepsilon}_L, \widetilde{\varepsilon}_R}$.

\begin{Prop}  \label{prop:notecpx}
The bilinearized differential $\partial^{\widetilde{\varepsilon}_L, \widetilde{\varepsilon}_R}$ of $\widetilde{\Lambda}^{(2N)}_{(k,m)}$ on the subcomplex
generated by the chords involving the component $\Lambda_0$ is given by
\begin{eqnarray*}
\partial^{\widetilde{\varepsilon}_L, \widetilde{\varepsilon}_R} c_{0,0} &=& 0, \\
\partial^{\widetilde{\varepsilon}_L, \widetilde{\varepsilon}_R} c_{0,j} &=& M_{0,j} + \overline{j} \ c_{0,j-1}, \\
\partial^{\widetilde{\varepsilon}_L, \widetilde{\varepsilon}_R} c_{j,0} &=& \overline{j} \ c_{j+1,0}, \\
\partial^{\widetilde{\varepsilon}_L, \widetilde{\varepsilon}_R} m_{0,j} &=& \overline{j} \ m_{0,j-1}, \\
\partial^{\widetilde{\varepsilon}_L, \widetilde{\varepsilon}_R} M_{0,j} &=& \overline{j} \ M_{0,j-1}, 
\end{eqnarray*}
for $j = 1, \ldots, k$, where $\overline{j}$ is the modulo 2 reduction of $j$ and where in the right hand sides $c_{k+1,0}, c_{0,0}, m_{0,0}$ and 
$M_{0,0}$ should be replaced by zero.
\end{Prop}

\begin{proof}
This result follows from the same computations as in Proposition~\ref{prop:bilinN}, in which we replace $2N$ with $k$,
$i$ with $0$ and where all terms obtained by changing the index $i$ are omitted since the mixed Reeb chords involving
$\Lambda_0$ are not augmented.
\end{proof}

\begin{Prop} \label{prop:noteN}
Consider the Legendrian link $\widetilde{\Lambda}^{(2N)}_{(k,m)} \subset J^1(\R^n)$. Its Poincar\'e polynomial 
with respect to the augmentations $\widetilde{\varepsilon}_L$ and $\widetilde{\varepsilon}_R$  is given by 
$$
P_{\widetilde{\Lambda}^{(2N)}_{(k,m)}, \widetilde{\varepsilon}_L, \widetilde{\varepsilon}_R}(t) =
1 + t^n +  t^{-m} + t^a,
$$
where
\begin{equation} \label{eq:exponent}
a = \left\{ \begin{array}{ll}
k-m-1 & \textrm{if } k \textrm{ is even,} \\
n-k+m &  \textrm{if } k \textrm{ is odd.} 
\end{array} \right.
\end{equation}
\end{Prop}

\begin{proof}
Let us compute the homology of the subcomplex generated by all Reeb chords involving the component $\Lambda_0$.
First note that $c_{0,0}$ is always a generator in homology, leading to the term $t^n$ in the Poincar\'e polynomial. Moreover, 
the complex generated by the chords $c_{0,1}, \ldots, c_{0,k}$ and $M_{0,1}, \ldots, M_{0,k}$ is acyclic.

If $k$ is even, the complex generated by the chords $c_{1,0}, \ldots, c_{k,0}$ is acyclic. On the other hand, the complex 
generated by the chords $m_{0,1}, \ldots, m_{0,k}$ has its homology generated by $m_{0,1}$ and $m_{0,k}$. These lead 
to the terms $t^{-m}$ and $t^{k-m-1}$ in the Poincar\'e polynomial.

If $k$ is odd, the complex generated by the chords $c_{1,0}, \ldots, c_{k,0}$ has its homology generated by $c_{k,0}$.
This leads to the term $t^{n-k+m}$ in the Poincar\'e polynomial. On the other hand, the complex 
generated by the chords $m_{0,1}, \ldots, m_{0,k}$ has its homology generated by $m_{0,1}$. This leads to the term $t^{-m}$ 
in the Poincar\'e polynomial.

Adding these contributions to the Poincar\'e polynomial of $\widetilde{\Lambda}^{(2N)}$ from Proposition~\ref{prop:polyNtilde},
we obtain the announced result.
\end{proof}

\begin{Rem}  \label{rmk:variantnote}
As a variant of the above construction, if we choose $\Lambda_0$ to be unlinked with $\Lambda_1$ in addition to
$\Lambda_{k+1}, \ldots, \Lambda_{2N}$, then we obtain instead the Poincar\'e polynomial $1 + t^n +  t^{n+m-2} + t^a$
with the same $a$ as in Proposition~\ref{prop:noteN}. This is because the subcomplex generated by all Reeb chords 
involving the component $\Lambda_0$ considered in the above proof does not contain the generators 
$c_{1,0}$ and $m_{0,1}$ anymore. Therefore, when $k$ is even its homology is generated by $c_{2,0}$ and $m_{0,k}$,
and when $k$ is odd it is generated by $c_{2,0}$ and $c_{k,0}$. Hence, in the Poincar\'e polynomial the exponent
$-m = |m_{0,1}|$ is replaced with $n+m-2 = |c_{2,0}|$.
\end{Rem}

The next step in our construction is to perform a connected sum between the component $\Lambda_0$ and the original knot 
$\widetilde{\Lambda}^{(2N)}$. This can be done after a Legendrian isotopy of $\Lambda_0$ similar to the one depicted in 
Figure~\ref{fig:connsum_highdim}, so that a piece of cusp in the deformed $\Lambda_0$ faces a piece of cusp from the component 
$\Lambda_1$. In this case, it will be necessary to use a different number of first Reidemeister moves as in Figure~\ref{fig:trefoilnote} before spinning 
the resulting front, so that the Maslov potentials near the facing cusps agree. We denote by $\overline{\Lambda}^{(2N)}_{(k,m)}$ the resulting 
Legendrian knot in $J^1(\R^n)$. We denote by $\overline{\varepsilon}_L$ and $\overline{\varepsilon}_R$ the augmentations induced from 
$\widetilde{\varepsilon}_L$ and $\widetilde{\varepsilon}_R$ via the exact Lagrangian cobordism between $\overline{\Lambda}^{(2N)}_{(k,m)}$ and 
$\widetilde{\Lambda}^{(2N)}_{(k,m)}$.

\begin{Prop} \label{prop:noteNconn}
Consider the Legendrian knot $\overline{\Lambda}^{(2N)}_{(k,m)} \subset J^1(\R^n)$. We have
$$
P_{\overline{\Lambda}^{(2N)}_{(k,m)}, \overline{\varepsilon}_L, \overline{\varepsilon}_R}(t) = 1 + t^{-m} + t^a,
$$
where $a$ is given by \eqref{eq:exponent}.
\end{Prop}

\begin{proof}
By Proposition~\ref{prop:noteN}, the generator 
$[c_{0,0}] \in LCH_n^{\widetilde{\varepsilon}_L, \widetilde{\varepsilon}_R}(\widetilde{\Lambda}^{(2N)}_{(k,m)})$
corresponds to the fundamental class $[\Lambda_0]$ of the component $\Lambda_0$ of the Legendrian link 
$\widetilde{\Lambda}^{(2N)}_{(k,m)}$. By Proposition~\ref{prop:connsum}, the effect of the connected sum with
this component is to remove the term $t^n$ from the Poincar\'e polynomial, so that we obtain the announced result.
\end{proof}

Note that, instead of adding a single component $\Lambda_0$ to the Legendrian knot $\widetilde{\Lambda}^{(2N)}$, we can add
a collection of components $\Lambda_{0,1}, \ldots, \Lambda_{0,r} \subset J^1(\R^n)$ with similar properties. More precisely, for 
all $i = 1, \ldots, r$, $\Lambda_{0,i}$ forms with the bottom $k_i$ components $\Lambda_1, \ldots, \Lambda_{k_i}$ a Legendrian 
link isotopic to the $k_i+1$-copy of the standard Legendrian unknot, but the projection of $\Lambda_{0,i}$  to $\R^n$ is disjoint from
the projection of the other components $\Lambda_{k_i+1}, \ldots, \Lambda_{2N}$. The Maslov potential of $\Lambda_{0,i}$ is fixed as
the Maslov potential of $\Lambda_1$ plus $m_i - 1$, for some integer $m_i$. With $\overline{k} = (k_1, \ldots, k_r)$ and $\overline{m} 
= (m_1, \ldots, m_r)$, we denote the resulting Legendrian link by $\widetilde{\Lambda}^{(2N)}_{(\overline{k},\overline{m})}$.

The addition of $\Lambda_{0,1}, \ldots, \Lambda_{0,r}$ to $\widetilde{\Lambda}^{(2N)}$  is illustrated by Figure~\ref{fig:partitionseveralnotes} 
in the case $r=3$ and $\{ k_1, k_2, k_3 \} = \{ 1, 3, 4 \}$, where the picture zooms on the bottom strata of the $k$ 
components $\Lambda_1, \ldots, \Lambda_k$, which are represented as portions of horizontal planes.

\begin{figure}
  \centerline{\includegraphics[width=7cm]{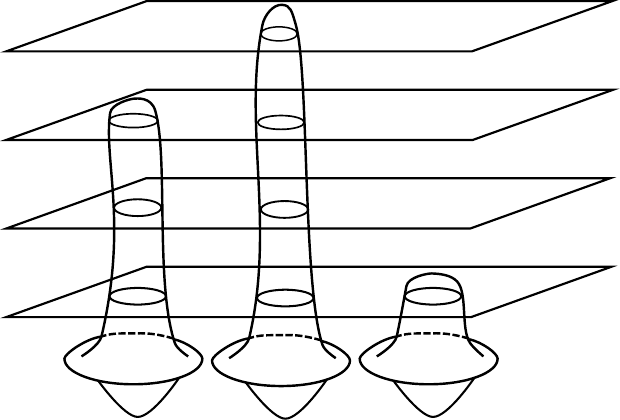}}
  \caption{Additional components $\Lambda_{0,i}$ with $r=3$ and $\{ k_1, k_2, k_3 \} = \{ 1, 3, 4 \}$.}
  \label{fig:partitionseveralnotes}
\end{figure}

Each additional component $\Lambda_{0,i}$ gives rise to an additional subcomplex in the bilinearized complex as in 
Proposition~\ref{prop:notecpx}, hence to additional terms in the Poincar\'e polynomial of the form $t^n + t^{-m_i} + t^{a_i}$
with $a_i$ given by~\eqref{eq:exponent}. After the connected sum of these components with $\widetilde{\Lambda}^{(N)}$,
we obtain a Legendrian knot $\overline{\Lambda}^{(2N)}_{(k,m)}$ and, arguing as in Proposition~\ref{prop:noteNconn}, its
Poincar\'e polynomial is given by
\begin{equation}  \label{eq:polysphere}
P_{\overline{\Lambda}^{(2N)}_{(\overline{k},\overline{m})}, \overline{\varepsilon}_L, \overline{\varepsilon}_R}(t) = 
1 + \sum_{i=1}^r (t^{-m_i} + t^{a_i}).
\end{equation}

At this point of our constructions we realized the geography of bLCH for Legendrian spheres $\Lambda$.

\begin{Thm}  \label{thm:geogsphere}
Let $P = q+p$ be the sum of Laurent polynomials with nonnegative integral coefficients satisfying conditions
(i') and (ii') from Remark~\ref{rmk:admsphere}. Then there exists a Legendrian sphere $\Lambda$ in $J^1(\R^n)$ 
and two non DGA homotopic augmentations $\varepsilon_1, \varepsilon_2$ of the Chekanov-Eliashberg DGA of $\Lambda$, 
with the property that the Poincar\'e polynomial of $LCH^{\varepsilon_1, \varepsilon_2}(\Lambda)$ with coefficients 
in $\Z_2$ is equal to $P$.
\end{Thm}

\begin{proof}
Let us show that the Poincar\'e polynomials obtained in \eqref{eq:polysphere} realize all polynomials 
$P = q + p$ satisfying conditions (i') and (ii').

Indeed, let $q(t)=1$ and $p$ be a Laurent polynomial satisfying the above condition (ii').
If $n$ is even, $p(-1)=0$ so that the polynomial $p$ can be expressed as a sum of polynomials of the form 
$\sum_{i=1}^r (t^{u_i} + t^{v_i})$, where $u_i < v_i$ have different parities. If $n$ is odd, $p(-1)$ is even, so that 
the polynomial $p$ can be expressed as the sum of polynomials of the form $\sum_{i=1}^r (t^{u_i} + t^{v_i})$, with no 
parity conditions on $u_i$ and $v_i$. 

In order to realize the polynomial $t^{u_i} + t^{v_i}$ when $u_i$ and $v_i$ have different parities, 
we can choose $m_i = - u_i$ and $k_i = v_i - u_i + 1$, which is even.
When $u_i$ and $v_i$ have the same parity, which can happen only if $n$ is odd, we proceed as follows.
If $u_i + v_i \le n-1$, we can choose $m_i= - u_i$ and $k_i = n - u_i - v_i$, which is odd.
If $u_i + v_i \ge n-1$, we use the variant of the construction with $\Lambda_0$ described in Remark~\ref{rmk:variantnote}
with $m_i= u_i+2-n$ and $k_i = u_i+v_i+3-n$, which is even.

Let us define $\overline{k} = (k_1, \ldots, k_r)$ and $\overline{m}=(m_1, \ldots, m_r)$, and let $N$ be the smallest
even integer such that $k_i \le 2N$ for all $i = 1, \ldots, r$. Then, in view of \eqref{eq:polysphere}, the Legendrian sphere
$\overline{\Lambda}^{(2N)}_{(\overline{k},\overline{m})}$ satisfies 
$$
P_{\overline{\Lambda}^{(2N)}_{(\overline{k},\overline{m})}, \overline{\varepsilon}_L, \overline{\varepsilon}_R}(t) = 1 + p(t)
= q(t) + p(t)
$$
as desired.
\end{proof}

%
%
\subsection{Geography of bLCH for general Legendrian submanifolds}

In order to obtain Poincar\'e polynomials with all possible polynomials $q$ satisfying condition (i) from Definition~\ref{def:bLCHadm}, 
we use the following construction from~\cite[Corollary~6.7]{BST}.

\begin{Prop}  \label{prop:embsurg}
For any monic polynomial $\overline{q}$ of degree $n$ satisfying $\overline{q}(0)=0$, there exists a connected 
Legendrian submanifold $\Lambda_{\overline{q}} \subset J^1(\R^n)$ equipped with an augmentation $\varepsilon$
such that $P_{\Lambda_{\overline{q}}, \varepsilon} = \overline{q}$.
\end{Prop}

If $q$ is a polynomial satisfying condition (i) from Definition~\ref{def:bLCHadm}, then the polynomial $\overline{q}$ given by
$\overline{q}(t) = q(t) + t^n - 1$ satisfies the assumptions of Proposition~\ref{prop:embsurg}.

Let $\Lambda^{(2N)}_{\overline{q}, (\overline{k},\overline{m})}$ be the disjoint union of the Legendrian knots 
$\overline{\Lambda}^{(2N)}_{(\overline{k},\overline{m})}$ and $\Lambda_{\overline{q}}$,
such that the projection of these components to $\R^n$ are disjoint. We denote by $\widehat{\varepsilon}_L$ and $\widehat{\varepsilon}_R$ the
augmentations for $\Lambda^{(2N)}_{\overline{q}, (\overline{k},\overline{m})}$ induced by the augmentation $\varepsilon$ for 
$\Lambda_{\overline{q}}$ and the augmentations $\overline{\varepsilon}_L$
and $\overline{\varepsilon}_R$ for $\overline{\Lambda}^{(2N)}_{(\overline{k},\overline{m})}$. The Poincar\'e polynomial of 
$\Lambda^{(2N)}_{\overline{q}, (\overline{k},\overline{m})}$ is given by the sum of the Poincar\'e polynomials of its components:
$$
P_{\Lambda^{(2N)}_{\overline{q}, (\overline{k},\overline{m})}, \overline{\varepsilon}_L, \overline{\varepsilon}_R}(t) 
= t^n + q(t) + \sum_{i=1}^r (t^{-m_i} + t^{a_i}).
$$
We then perform a connected sum on the Legendrian link $\Lambda^{(2N)}_{\overline{q}, (\overline{k},\overline{m})}$ 
in order to obtain a Legendrian knot $\widetilde{\Lambda}^{(2N)}_{\overline{q}, (\overline{k},\overline{m})}$,
equipped with two augmentations still denoted by $\widehat{\varepsilon}_L$ and $\widehat{\varepsilon}_R$. Since the augmentations 
$\overline{\varepsilon}_L$ and  $\overline{\varepsilon}_R$ coincide (with $\varepsilon$) on the component $\Lambda_{\overline{q}}$, 
by Proposition~\ref{prop:tau}
the fundamental class $[\Lambda_{\overline{q}}]$ of this component is in the image of the map $\tau_n$ in the duality exact sequence~\eqref{eq:bilindualityseq}. By Proposition~\ref{prop:connsum}, the effect of the connected sum with $\Lambda_{\overline{q}}$ is to remove 
a term $t^n$ from the Poincar\'e polynomial. We therefore obtain 
$$
P_{\Lambda^{(2N)}_{\overline{q}, (\overline{k},\overline{m})}, \overline{\varepsilon}_L, \overline{\varepsilon}_R}(t) 
= q(t) + \sum_{i=1}^r (t^{-m_i} + t^{a_i}).
$$

Although these Poincar\'e polynomials realize all polynomials $q$ satisfying condition (i) from Definition~\ref{def:bLCHadm}, we are still missing
some Laurent polynomials $p$, since these can be arbitrary when $n > 2$. In order to realize these more general Laurent polynomials $p$, we
 describe a generalization of the embedded surgery construction on which Proposition~\ref{prop:embsurg} and its proof in~\cite[Corollary~6.7]{BST}
 are based.

From now on, assume that $n \ge 2$.
Consider a point on the cusp locus of the component $\Lambda_1$ of the $2N$-copy of the standard Legendrian unknot 
$\Lambda^{(2N)} \subset J^1(\R^n)$. By a Legendrian isotopy, it is always possible to arrange so that in a neighborhood
of this point, the front of $\Lambda^{(2N)}$ in $J^0(\R^n)$ with local coordinates $(x_1, \ldots, x_n, z)$ is locally described 
as follows: the fragment of $\Lambda_1$ in this neighborhood is composed of a bottom stratum $z=0$ and of a top stratum
satisfying $z^2 = x_n^3$, both for $x_n \ge 0$. Moreover, the fragments of the bottom strata of the components $\Lambda_i$ 
in this neighborhood satisfy $z = (i-1)\varepsilon$ for $i=2, \ldots, 2N$, and no other parts of the front of $\Lambda^{(2N)}$ lie 
in this neighborhood. Note that it is possible to arrange so that this local model still holds
for the more sophisticated Legendrian $\Lambda^{(2N)}_{\overline{q}, (\overline{k},\overline{m})}$ after our above constructions.

For a given $m' \in \{ 0, \ldots, n-2 \}$, we consider an embedded sphere $S^{m'}$ of dimension $m'$ in the cusp locus 
$\{ x_n = z = 0 \}$ of $\Lambda_1$. In view of our assumptions on the front of $\Lambda^{(2N)}$, this sphere bounds an
embedded disk of dimension $m'+1$ with its interior disjoint from the front of $\Lambda^{(2N)}$.
For a given $k' \in \{ 2, \ldots, 2N \}$, we define a function $f$ on the cusp locus of
$\Lambda_1$, equal to $((k'+\frac23)\varepsilon)^{2/3}$ along $S^{m'}$, given by 
$\frac{((k'+\frac23)\varepsilon)^{2/3}}{r_0^{1/2}}\sqrt{r_0-r}$ 
at distance $r \in (0, r_0]$ from $S^{m'}$ and extended by $0$ everywhere else. We remove from the front of $\Lambda_1$
the region satisfying $x_n < f(x_1, \ldots, x_{n-1})$; the resulting front has boundary diffeomorphic to the cartesian product
of $S^{m'}$ with a standard Legendrian sphere of dimension $n-m'-1$, with a flat bottom stratum. We now perform an $m'$-surgery 
on $\Lambda^{(2N)}$ by attaching a standard Legendrian handle diffeomorphic to $D^{m'+1} \times S^{n-m'-1}$ to the above front 
along its boundary.
By construction, along the boundary of this handle, the standard Legendrian sphere of dimension $n-m'-1$ has height 
$(k'+\frac23)\varepsilon$. We shape the handle so that this height decreases monotonically from the boundary of
$D^{m'+1}$ to its center, where it takes the minimal value $(k'+\frac13)\varepsilon$. This is a standard Legendrian
surgery on $\Lambda_1$, but it is of a more general nature if we consider the whole $\Lambda^{(2N)}$, since the front of the
attached handle intersects the front of the components $\Lambda_2, \ldots, \Lambda_{k'+1}$ (but not of the components 
$\Lambda_{k'+2}, \ldots, \Lambda_{2N}$). When this operation is performed on the Legendrian submanifold 
$\Lambda^{(2N)}_{\overline{q}, (\overline{k},\overline{m})}$, we denote the resulting Legendrian submanifold by 
$\Lambda^{(2N)}_{\overline{q}, (\overline{k},\overline{m}), (k',m')}$.

\begin{figure}
\labellist
\small\hair 2pt
\pinlabel {$\Lambda_1$} [bl] at 33 0
\pinlabel {$\Lambda_2$} [bl] at -10 35
\pinlabel {$\Lambda_j$} [bl] at -10 80
\pinlabel {$\Lambda_{k'+1}$} [bl] at -10 125
\pinlabel {$c'_{1,1}$} [bl] at 140 98
\pinlabel {$c'_{1,j}$} [bl] at 90 56
\pinlabel {$c'_{j,1}$} [bl] at 100 98
\endlabellist
  \centerline{\includegraphics[width=7cm]{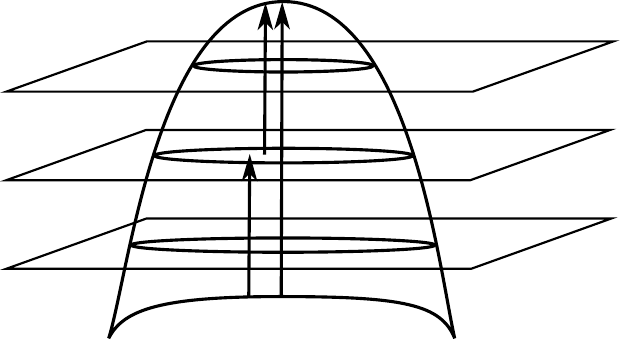}}
  \caption{Center of a generalized handle.}
  \label{fig:slicehandle}
\end{figure}

In order to minimize the number of Reeb chords created by this operation, we shape the standard Legendrian sphere of 
dimension $n-m'-1$ as shown in Figure~\ref{fig:slicehandle}, with both top and bottom strata being the graphs of concave
functions. Assuming for simplicity that the minima of the perturbing Morse functions $f_i-f_j$ for $i \neq j$ are located in the 
bottom strata and that the corresponding maxima are located in the top strata, the bottom strata of the $\Lambda_i$ are 
slightly moving away from each other in the $z$-direction as $x_n$ decreases to $0$. Hence, the bottom stratum of the 
standard Legendrian sphere of 
dimension $n-m'-1$ is slightly moving down from the boundary of $D^{m'+1}$ to its center. In particular, all new Reeb chords
are located very close to the center of the handle: $c'_{1,1}$ with endpoints on the handle, $c'_{1,j}$ from the handle to 
$\Lambda_j$ and $c'_{j,1}$ from $\Lambda_j$ to the handle, for $j = 2, \ldots, k'+1$, as shown in Figure~\ref{fig:slicehandle}.
On the other hand, we can perturb the resulting Legendrian submanifold so that there are no Reeb chords between the 
attached handle and the components $\Lambda_{k'+2}, \ldots, \Lambda_{2N}$.
Summarizing, the gradings and lengths of the new Reeb chords are given by
\begin{center}
\begin{tabular}{c|cc}
& grading & length \\
\hline
$c'_{1,1}$ & $n-m'-1$ & $(k'+\frac13)\varepsilon$ \\
$c'_{j,1}$ & $n-m'-j$ & $(k'-j+\frac43)\varepsilon$ \\
$c'_{1,j}$ & $m'+j-1$ & $(j-1)\varepsilon$ 
\end{tabular}
\end{center}

\begin{Prop}  \label{prop:diffhandle}
The augmentations $\widehat{\varepsilon}_L$ and $\widehat{\varepsilon}_R$ can be extended by zero on the new chords
to augmentations of $\Lambda^{(2N)}_{\overline{q}, (\overline{k},\overline{m}), (k',m')}$. The vector space spanned by the
new chords $c'_{1,1}$, $c'_{1,j}$ and $c'_{j,1}$ is a subcomplex with respect to the bilinearized differentials 
$\partial^{\widehat{\varepsilon}_L, \widehat{\varepsilon}_R}$ and $\partial^{\widehat{\varepsilon}_R, \widehat{\varepsilon}_L}$.
These differentials are given by
\begin{eqnarray*}
\partial^{\widehat{\varepsilon}_L, \widehat{\varepsilon}_R} c'_{1,j+1} &=& \overline{j+1} \ c'_{1,j}, \\
\partial^{\widehat{\varepsilon}_L, \widehat{\varepsilon}_R} c'_{j,1} &=& \overline{j} \ c'_{j+1,1}, \\
\partial^{\widehat{\varepsilon}_L, \widehat{\varepsilon}_R} c'_{k'+1,1} &=& 0,
\end{eqnarray*}
and respectively by
\begin{eqnarray*}
\partial^{\widehat{\varepsilon}_R, \widehat{\varepsilon}_L} c'_{1,j+1} &=& 
\left\{ \begin{array}{ll} 
\overline{j} \ c'_{1,j} & \textrm{if } j \neq 1, \\
0 & \textrm{if } j = 1,
\end{array} \right. \\
\partial^{\widehat{\varepsilon}_R, \widehat{\varepsilon}_L} c'_{j,1} &=& \overline{j+1} \ c'_{j+1,1}, \\
\partial^{\widehat{\varepsilon}_R, \widehat{\varepsilon}_L} c'_{k'+1,1} &=& 0,
\end{eqnarray*}
for $j = 1, \ldots, k'$, where $\overline{j}$ is the modulo $2$ reduction of $j$.
\end{Prop}

\begin{proof}
We first need to show that $\widehat{\varepsilon}_L \circ \partial c = \widehat{\varepsilon}_R \circ \partial c = 0$ 
for any Reeb chord $c$ of $\Lambda^{(2N)}_{\overline{q}, (\overline{k},\overline{m}), (k',m')}$. If $c$ is a Reeb
chord of $\Lambda^{(2N)}_{\overline{q}, (\overline{k},\overline{m})}$, then $\partial c$ consists of terms from
the differential for $\Lambda^{(2N)}_{\overline{q}, (\overline{k},\overline{m})}$, hence in the kernel of 
$\widehat{\varepsilon}_L$ and $\widehat{\varepsilon}_R$, and of terms involving at least one new chord of 
$\Lambda^{(2N)}_{\overline{q}, (\overline{k},\overline{m}), (k',m')}$. Since $\widehat{\varepsilon}_L$ and 
$\widehat{\varepsilon}_R$ vanish on these new chords, we obtain the desired relations.

If $c$ is a new chord of $\Lambda^{(2N)}_{\overline{q}, (\overline{k},\overline{m}), (k',m')}$, we claim that
any term in $\partial c$ contains an unaugmented chord as a factor, and hence is in the kernel of 
$\widehat{\varepsilon}_L$ and $\widehat{\varepsilon}_R$. Indeed, the only augmented chords go from
$\Lambda_j$ to $\Lambda_{j+1}$, with a parity condition on $j$ depending on the augmentation. Moreover,
Morse flow trees cannot entirely go across a connecting tube (since they are attracted to its center) so that
chords are the only way to jump from $\Lambda_i$ to $\Lambda_j$ with $i \neq j$. Since the new chords have
at least one endpoint on $\Lambda_1$, if a Morse flow tree has all negative ends at augmented chords, 
it must start at $c'_{1,1}$ or at $c'_{1,2}$. But $|c'_{1,1}| = n-m'-1$ equals $1$ if and only if $m'=n-2$, and
in that case a Morse flow tree with endpoints remaining on $\Lambda_1$ must remain in the center of the
handle, which is a $1$-dimensional standard Legendrian knot, so that there are $2$ such Morse flow trees
with no negative end, canceling each other. On the other hand, $|c'_{1,2}| = m'+1$ equals $1$ if and only if 
$m'=0$, and in that case a Morse flow tree with endpoints remaining on $\Lambda_1$ and $\Lambda_2$ must
connect the critical point $c'_{1,2}$ of $f_2 - f_1$ of index $1$ to the critical point $m_{1,2}$ of $f_2 - f_1$ of 
index $0$. There are $2$ such Morse flow trees, corresponding to the $2$ sides of the $1$-dimensional 
unstable manifold of $c'_{1,2}$, and these cancel each other.

Let us now compute the bilinearized differentials.
If a rigid Morse flow tree starting at $c'_{1,j}$ with $j = 1, \ldots, k'+1$, has only one negative end, it will leave 
the handle radially and then flow to the minimum $m_{1,j}$ of $f_j - f_1$. Such a configuration is rigid if and 
only if $|m_{1,j}| = j-2 = |c'_{1,j}|-1= m' + j-2$, but when $m'=0$ there are $2$ such Morse flow trees as above, 
canceling each other. If it has more negative ends and contributes to the bilinearized differential of
$c'_{1,j}$, it can only have a negative end at $m_{j-1,j}$, 
and the other one must then be at $c'_{1,j-1}$. There is a unique such Morse flow tree, flowing from $c'_{1,j}$ to
the position of $c'_{1,j-1}$ in the $D^{m'+1}$-factor of the handle, then splitting at the bottom stratum of 
$\Lambda_{j-1}$, so that one part flows in the $S^{n-m'-1}$-factor of the handle to $c'_{1,j-1}$ and the other 
part flows to the minimum $m_{j-1,j}$ of $f_j - f_{j-1}$. This term $m_{j-1,j} c'_{1,j-1}$
gives rise to the term $c'_{1,j-1}$ in $\partial^{\widehat{\varepsilon}_L, \widehat{\varepsilon}_R} c'_{1,j}$
if and only if $\widehat{\varepsilon}_L (m_{j-1,j}) = 1$, i.e. when $j$ is odd and $> 1$. It gives rise to the
term $c'_{1,j-1}$ in $\partial^{\widehat{\varepsilon}_R, \widehat{\varepsilon}_L} c'_{1,j}$ if and only if 
$\widehat{\varepsilon}_R (m_{j-1,j}) = 1$, i.e. when $j$ is even.

Let us now consider a rigid Morse flow tree starting at $c'_{j,1}$ with $j = 2, \ldots, k'+1$. Such a tree cannot
have only one negative end, and if it contributes to the bilinearized differential of $c'_{j,1}$,
it must have two negative ends, one at $m_{j,j+1}$ and the other one at $c'_{j+1,1}$. There is a unique such 
Morse flow tree, flowing from $c'_{j,1}$ to
the position of $c'_{j+1,1}$ in the $S^{n-m'-1}$-factor of the handle, then splitting at the bottom stratum of 
$\Lambda_{j+1}$, so that one part flows in the $D^{m'+1}$-factor of the handle to $c'_{j+1,1}$ and the other 
part flows to the minimum $m_{j,j+1}$ of $f_{j+1} - f_j$. 
This term $c'_{j+1,1} m_{j,j+1}$ gives rise to the term $c'_{j+1,1}$ in 
$\partial^{\widehat{\varepsilon}_L, \widehat{\varepsilon}_R} c'_{j,1}$
if and only if $\widehat{\varepsilon}_R (m_{j,j+1}) = 1$, i.e. when $j$ is odd and $< k'+1$. It gives rise to the
term $c'_{j+1,1}$ in $\partial^{\widehat{\varepsilon}_R, \widehat{\varepsilon}_L} c'_{j,1}$ if and only if 
$\widehat{\varepsilon}_L (m_{j,j+1}) = 1$, i.e. when $j$ is even and $< k'+1$.
\end{proof}

As an immediate consequence of Proposition~\ref{prop:diffhandle}, the homology with respect to 
$\partial^{\widehat{\varepsilon}_L, \widehat{\varepsilon}_R}$ of the subcomplex 
generated by the new Reeb chords is generated by $[c'_{k'+1,1}]$ in degree $n-m'-k'-1$ if $k'$ is even, 
and by $[c'_{1,k'+1}]$ in degree $m'+k'$ if $k'$ is odd. Similarly, the homology with respect to 
$\partial^{\widehat{\varepsilon}_R, \widehat{\varepsilon}_L}$ of this subcomplex is generated by 
$[c'_{1,1}]$ in degree $n-m'-1$, $[c'_{1,2}]$ in degree $m'+1$, and by $[c'_{1,k'+1}]$ in degree 
$m'+k'$ if $k'$ is even, and by $[c'_{k'+1,1}]$ in degree $n-m'-k'-1$ if $k'$ is odd.

\begin{Prop}  \label{prop:effecthandle}
The bLCH Poincar\'e polynomials of $\Lambda^{(2N)}_{\overline{q}, (\overline{k},\overline{m}), (k',m')}$ are given by
$$
P_{\Lambda^{(2N)}_{\overline{q}, (\overline{k},\overline{m}), (k',m')}, \widehat{\varepsilon}_L, \widehat{\varepsilon}_R}(t) =
P_{\Lambda^{(2N)}_{\overline{q}, (\overline{k},\overline{m})}, \widehat{\varepsilon}_L, \widehat{\varepsilon}_R}(t)
+ t^{b}
$$
and by
$$
P_{\Lambda^{(2N)}_{\overline{q}, (\overline{k},\overline{m}), (k',m')}, \widehat{\varepsilon}_R, \widehat{\varepsilon}_L}(t) =
P_{\Lambda^{(2N)}_{\overline{q}, (\overline{k},\overline{m})}, \widehat{\varepsilon}_R, \widehat{\varepsilon}_L}(t)
+ t^{n-m'-1} + t^{m'+1} + t^{n-1-b}
$$
where $b = n-m'-k'-1$ if $k'$ is even and $b = m'+k'$ if $k'$ is odd.
\end{Prop}

\begin{proof}
Observe that the image of $[c'_{1,1}]$ by the map 
$$
\tilde{\tau}_{n-m'-1} : LCH_{n-m'-1}^{\widehat{\varepsilon}_R, \widehat{\varepsilon}_L}
(\Lambda^{(2N)}_{\overline{q}, (\overline{k},\overline{m}), (k',m')}) \to H_{n-m'-1}(\Lambda^{(2N)}_{\overline{q}, (\overline{k},\overline{m}), (k',m')})
$$ 
from the duality exact sequence~\eqref{eq:bilindualityseq2} is the homology class of the co-core sphere of the 
attached handle. Indeed, all Morse flow trees starting at $c'_{1,1}$ and with no negative end must remain in the 
co-core sphere of the handle, since the latter is narrowest there. The resulting Morse flow trees start at $c'_{1,1}$
in any direction and finish at the cusp of the co-core sphere. The boundary of the corresponding holomorphic disks 
foliate the co-core sphere minus the endpoints of $c'_{1,1}$ so that the image of the cycle $c'_{1,1}$ in the bilinearized 
complex is the cycle corresponding to the co-core sphere in the singular complex of 
$\Lambda^{(2N)}_{\overline{q}, (\overline{k},\overline{m}), (k',m')}$. 
Since the corresponding homology class does not vanish in 
$H_{n-m'-1}(\Lambda^{(2N)}_{\overline{q}, (\overline{k},\overline{m}), (k',m')})$, it follows that $[c'_{1,1}]$ does not vanish
in bilinearized homology either.

Similarly, observe that the image of $[c'_{1,2}]$ by the map 
$$
\tilde{\tau}_{m'+1} : LCH_{m'+1}^{\widehat{\varepsilon}_R, \widehat{\varepsilon}_L}
(\Lambda^{(2N)}_{\overline{q}, (\overline{k},\overline{m}), (k',m')}) \to H_{m'+1}(\Lambda^{(2N)}_{\overline{q}, (\overline{k},\overline{m}), (k',m')})
$$ 
from the duality exact sequence~\eqref{eq:bilindualityseq2} is the Poincar\'e dual of the homology class of the co-core 
sphere of the attached handle. Indeed, all Morse flow trees starting at $c'_{1,2}$ and with no negative end must follow
radii of the disk factor $D^{m'+1}$ for the handle. Once such a Morse flow tree exits the handle, it will flow down to the
chord $m_{1,2}$ corresponding the the minimum of the perturbing Morse function $f_2-f_1$. The chord $m_{1,2}$ is 
augmented for $\widehat{\varepsilon}_R$ so that the image by $\tilde{\tau}_{m'+1}$ is obtained by considering the part 
of the boundary of the corresponding holomorphic disks lying in $\Lambda_1$. This is a sphere of dimension $m'+1$,
intersecting the co-core sphere at the endpoint of $c'_{1,2}$ in $\Lambda_1$. Since the corresponding homology class 
does not vanish in $H_{m'+1}(\Lambda^{(2N)}_{\overline{q}, (\overline{k},\overline{m}), (k',m')})$, it follows that $[c'_{1,2}]$ 
does not vanish in bilinearized homology either.

In view of the long exact sequence relating the bilinearized homology of our subcomplex with the bilinearized homologies
of our Legendrian submanifold before and after the generalized handle attachment, the effect of $[c'_{k'+1,1}]$ or
$[c'_{1,k'+1}]$ could either be to add a term in the bLCH Poincar\'e polynomial in the degree of this generator, or to
remove a term in the degree of this generator, plus one.

In terms of Proposition~\ref{prop:relpoly}, we have just shown that the polynomial $\tilde{q}$ gains the terms 
$t^{n-m'-1} + t^{m'+1}$ as an effect of this generalized handle attachment. Since the dimension of the singular 
homology of the Legendrian submanifold increased by $2$, it follows that the modifications due to $[c'_{k'+1,1}]$ 
and $[c'_{1,k'+1}]$ are affecting the polynomials $p$ and $\tilde{p}$. Since the relation $\tilde{p}(t) = t^{n-1} p(t^{-1})$
must hold at all times, it follows that the changes to both bLCH Poincar\'e polynomials must occur in degrees that add
up to $n-1$. But since the sum of the gradings of $[c'_{k'+1,1}]$ and of $[c'_{1,k'+1}]$ is $n-1$, it follows that the effect 
of these generators is necessarily to add a term in their corresponding bLCH Poincar\'e polynomial.

Since the $4$ generators $[c'_{1,1}]$, $[c'_{1,2}]$, $[c'_{k'+1,1}]$ and $[c'_{1,k'+1}]$ each give rise to an additional
term in one of the bLCH Poincar\'e polynomials of $\Lambda^{(2N)}_{\overline{q}, (\overline{k},\overline{m}), (k',m')}$,
the announced relations follow.
\end{proof}

We can repeat the above generalized handle attachment as many times as we want, with different values of $k'$ 
and $m'$. Repeating it $s$ times with parameters $k'_i$ and $m'_i$, let us define $\overline{k'} = (k'_1, \ldots, k'_s)$ 
and $\overline{m'} = (m'_1, \ldots, m'_s)$, and after choosing $N$ so that $k'_i +1 \le 2N$ for all $i = 1, \ldots, s$.
Applying these operations on $\Lambda^{(2N)}_{\overline{q}, (\overline{k},\overline{m})}$, we denote the resulting 
Legendrian submanifold by $\Lambda^{(2N)}_{\overline{q}, (\overline{k},\overline{m}), (\overline{k'},\overline{m'})}$. 

\begin{Cor}  \label{cor:polyfinal}
The bLCH Poincar\'e polynomial of $\Lambda^{(2N)}_{\overline{q}, (\overline{k},\overline{m}), (\overline{k'},\overline{m'})}$ 
is given by
$$
P_{\Lambda^{(2N)}_{\overline{q}, (\overline{k},\overline{m}), (\overline{k'},\overline{m'})}, \widehat{\varepsilon}_L, \widehat{\varepsilon}_R}(t) 
= q(t) + \sum_{i=1}^r (t^{-m_i} + t^{a_i}) + \sum_{i=1}^s t^{b_i},
$$
where 
$$
a_i = \left\{ \begin{array}{ll}
k_i-m_i-1 & \textrm{if } k_i \textrm{ is even,} \\
n-k_i+m_i & \textrm{ if } k_i \textrm{ is odd,}
\end{array} \right.
$$
and 
$$
b_i =  \left\{ \begin{array}{ll}
n-k'_i-m'_i-1 & \textrm{if } k'_i \textrm{ is even,} \\
k'_i+m'_i & \textrm{ if } k'_i \textrm{ is odd.}
\end{array} \right.
$$
\end{Cor}

\begin{proof}[Proof of Theorem~\ref{thm:B}]
Note that if $n=1$, any connected Legendrian submanifold $\Lambda$ is a circle. Since we already showed that
the bLCH geography for spheres is realized by the submanifolds $\overline{\Lambda}^{(2N)}_{(\overline{k},\overline{m})}$
with Poincar\'e polynomial given by~\eqref{eq:polysphere} with $q(t) =1$ and $p(-1)$ even, we can assume that $n \ge 2$.

Assume first that $n > 2$. Let $q + p$ is a bLCH-admissible polynomial in the sense of Definition~\ref{def:bLCHadm}. 
Writing $p(t) = \sum_{i=1}^s t^{w_i}$, for any term $i  = 1, \ldots, s$, we can find $k'_i \ge 1$ and $0 \le m'_i \le n-2$ 
such that $b_i = w_i$ as in
Corollary~\ref{cor:polyfinal}: if $w_i > 0$ is odd, we can choose $m'_i=0$  and $k'_i = w_i$, if $w_i > 0$ is even, we
can choose $m'_i=1$ and $k'_i = w_i - 1$, if $w_i \le 0$ has the same parity as $n$, we can choose $m'_i=1$ and
$k'_i = n-2-w_i$, and if $w_i \le 0$ has the same parity as $n-1$, we can choose $m'_i=0$ and $k'_i = n-1-w_i$.
Then the Legendrian submanifold $\Lambda^{(2N)}_{\overline{q}, (\overline{k'},\overline{m'})}$ has the desired 
bLCH Poincar\'e polynomial $q+p$.

Finally, in the case $n=2$, we cannot use the above choices of parameters since we must have $m'_i =0$ for all 
$i = 1, \ldots, s$. Let $q + p$ is a bLCH-admissible polynomial in the sense of Definition~\ref{def:bLCHadm}. 
Let us decompose $p$ as $p_0 + p_1$ where $p_0$ and $p_1$ are Laurent polynomials with nonnegative integral 
coefficients, $p_0(-1)=0$ and $p_1(1)$ is minimal with respect to these properties.
We have already showed that there exists a Legendrian sphere $\overline{\Lambda}^{(2N)}_{(\overline{k},\overline{m})}$
with bLCH Poincar\'e polynomial given by $1 + p_0$ in view of~\eqref{eq:polysphere}. Since $p_1(1)$ is minimal, it follows
that all terms in $p_1$ have degrees of the same parity. 

If this parity is odd, all terms in $p_1$ are of the form $t^{w_i}$ with $w_i$ odd. If $w_i \ge 1$, we choose 
$k'_i = w_i$ odd, and if $w_i \le -1$, we choose $k'_i = 1- w_i$ even as in Corollary~\ref{cor:polyfinal}. 
Therefore, using as many generalized handle attachments as
needed, we can realize the bLCH Poincar\'e polynomial $1+ p_0 + p_1$, regardless of the value of $p_1(-1) \le 0$. 
Then, by a connected sum with the Legendrian submanifold $\Lambda_{\overline{q}}$ from Proposition~\ref{prop:embsurg},
we realize the bLCH Poincar\'e polynomial $q + p$ as desired.

If the terms in $p_1$ have degrees of even parity, we use generalized handle attachments on $\Lambda_2$ instead 
of $\Lambda_1$: the effect of this modified operation will be as described by Proposition~\ref{prop:effecthandle}, 
with the ordering of the augmentations reversed. In other words, each such generalized handle attachment will 
add $2t + t^{1-b_i}$ to the bLCH Poincar\'e polynomial of our Legendrian submanifold, with $1-b_i = 1 - k'_i$ even
as in Corollary~\ref{cor:polyfinal}. If $q(t) = 1 + at$ then we 
can perform up to $\lfloor \frac{a}2 \rfloor$ such attachments. Therefore, for any polynomial $p_1$ such that 
$p_1(-1) = p_1(1) \le \frac{a}2 = \frac12  (1 - q(-1))$, we can realize the bLCH Poincar\'e polynomial 
$1 + 2p_1(1) t + p_0 + p_1$. Setting $q_0(t) = q(t) - 2p_1(1) t$, we then perform a connected sum with the 
Legendrian submanifold $\Lambda_{\overline{q_0}}$ from Proposition~\ref{prop:embsurg}, in order to 
realize the bLCH Poincar\'e polynomial $q + p$ as desired.
\end{proof}

 \end{document}